\documentclass[11pt,reqno]{amsart}

\usepackage{fourier}
\usepackage{fullpage}

\usepackage{amsmath}
\usepackage{mathtools}
\usepackage{amssymb}
\usepackage{amsthm}
\usepackage{mathrsfs}
\usepackage{enumerate}
\newcounter{name}

\usepackage{xcolor}
\usepackage{hyperref}

\newcommand{\R}{\mathbb{R}}

\newcommand{\T}{\mathbb{T}}

\makeatletter
\renewcommand*\env@cases[1][1.2]{%
  \let\@ifnextchar\new@ifnextchar
  \left\lbrace
  \def\arraystretch{#1}%
  \array{@{}l@{\quad}l@{}}%
}
\makeatother

\newtheorem{lemma}{Lemma}[section]
\newtheorem{theorem}[lemma]{Theorem}
\newtheorem{proposition}[lemma]{Proposition}

\newtheorem{conjecture}[lemma]{Conjecture}
\theoremstyle{definition}

\theoremstyle{remark}
\newtheorem{remark}[lemma]{Remark}
\newtheorem{question}[lemma]{Question}

\title{Relaxation dynamics of the Inertial Winfree model}

\author[Moreno-Earle]{Caiman Moreno-Earle}
\address[Caiman Moreno-Earle]{\newline Mathematics Department \newline California Institute of Technology, Pasadena California 91125, United States}
\email{cmorenoearle@caltech.edu}

\author[Ryoo]{Seung-Yeon Ryoo }
\address[Seung-Yeon Ryoo]{\newline Mathematics Department \newline California Institute of Technology, Pasadena California 91125, United States}
\email{sryoo@caltech.edu}

\author[To]{Grace To}
\address[Grace To]{\newline Mathematics Department \newline California Institute of Technology, Pasadena California 91125, United States}
\email{gto@caltech.edu}
\date{\today}

\begin{document}


\begin{abstract}
We prove two synchronization theorems for the second-order (inertial) Winfree model of coupled oscillators. The first result is a pathwise oscillator-death theorem with explicit smallness thresholds on the natural frequencies, initial velocities, and inertia, scaling as $R_0^{3/2}$ in the initial order parameter $R_0$. The second result is a qualitative zero-inertia synchronization statement: under generic initial data, if the intrinsic and initial velocity spreads are small compared to $\kappa$ and the inertia $m$ is small, then the limiting order parameter can be made arbitrarily close to 2. The proof of the first result is organized around three mechanisms, namely inertial gradient flow and the {\L}ojasiewicz theorem, an initial layer argument, and an order-parameter bootstrapping argument. The proof of the second result involves approximation to the first-order case via a quantitative Tikhonov theorem.
\end{abstract}
\maketitle

\tableofcontents

\section{Introduction}\label{sec:intro}
\subsection{Formulation of the models and previous work}
\emph{Synchronization} refers to collective phenomena of many-body systems wherein some aspect, such as phase, position, or frequency, of the constituent particles or individuals coalesces. It is observed across a wide range of physical, biological, and engineering systems: following Huygens's seminal mid-seventeenth-century observation of two pendulum clocks hanging on a common bar \cite{oliveira2015huygens}, synchronous phenomena have been widely reported, including flocking in birds and fish \cite{flierl1999individuals,vicsek1995novel}, synchronous flashing of fireflies \cite{buck1966biology}, collective dynamics of cardiac pacemaker cells \cite{peskin1975mathematical}, and frequency synchronization in power networks \cite{chiang2011direct,kundur2007power,sauer2017power}. The systematic mathematical study of these phenomena began relatively recently with models such as the Winfree \cite{winfree1967biological} and Kuramoto \cite{kuramoto1975international} models of coupled oscillators and the Vicsek and Cucker--Smale models for flocking \cite{cucker2007emergent,vicsek1995novel}. The Winfree and Kuramoto models notably exhibit phase transitions from disordered (incoherent) states to partially locked and then to completely locked states as the \emph{coupling strength} exceeds certain critical thresholds \cite{ariaratnam2001phase,crawford1994amplitude,kuramoto2002coexistence}, a feature that has drawn sustained attention from the control theory, neuroscience, and statistical physics communities \cite{acebron2005kuramoto,bernoff2011primer,dorfler2011critical,ermentrout2019recent,ha2016collective,hoppensteadt2012weakly,rodrigues2016kuramoto,strogatz2000kuramoto}.

In this paper, we focus on the \emph{inertial Winfree model}, which to our knowledge was first introduced by \cite{ha2021emergent}. We will show below in Lemma \ref{lem:embedding}\footnote{This also shows that the first-order Winfree model embeds into the first-order Kuramoto model. We are not aware of prior observations of this fact.} that the inertial Winfree model embeds into the \emph{inertial Kuramoto model}, which is a second-order variant introduced by Arthur Bergen and David Hill \cite{bergen2007structure} to model electric networks with generators, and by Bard Ermentrout \cite{ermentrout1991adaptive} to model synchronous flashing of the firefly species Pteroptyx malaccae. Owing to its second-order nature, the inertial Kuramoto model exhibits novel features absent in the first-order Kuramoto model, including first-order phase transitions \cite{tanaka1997first} and hysteresis \cite{hong1999inertia,tanaka1997self}; by our embedding of Lemma \ref{lem:embedding}, these features carry over to the inertial Winfree model.

We now introduce the formal setup.
Let $N$ be the number of oscillators, and for each $1\le i\le N$, let
$\theta_i=\theta_i(t)\in\R$ denote the phase and $\omega_i=\dot\theta_i$
the instantaneous frequency of the $i$-th oscillator, where both are defined as real-valued functions of time $t\ge 0$. The evolution of the phase variables $\{\theta_i\}_{i=1}^N$ under the \emph{inertial
Winfree dynamics} is described by the Cauchy problem
\begin{equation}\label{InertWinfree}
    \begin{cases}
        m\ddot\theta_i + \dot{\theta}_i=\nu_i+\dfrac\kappa N \sum_{j=1}^N I(\theta_j(t))\,S(\theta_i(t)), & t>0, \\[2pt]
        (\theta_i,\dot\theta_i)\big\rvert_{t=0^+} =(\theta_i^0,\omega_i^0), &  i=1,\dots,N,
    \end{cases}
\end{equation}
where $I\colon\R\to\R$ (the \emph{influence function}) and
$S\colon\R\to\R$ (the \emph{sensitivity function}) are $2\pi$-periodic
Lipschitz functions, the constants $m\ge 0$ and $\kappa\ge 0$ denote
the (uniform) inertia and coupling strength, and $\nu_i\in\R$ denotes
the natural frequency of the $i$-th oscillator.

In comparison, the first-order Winfree model is formally obtained from
\eqref{InertWinfree} by setting $m=0$:\footnote{The formal justification of obtaining \eqref{GenWinfree} from \eqref{InertWinfree} is given in Proposition \ref{prop:mto0quant}.}
\begin{equation}\label{GenWinfree}
    \begin{cases}
        \dot{\theta}_i=\nu_i+\dfrac\kappa N \sum_{j=1}^N I(\theta_j(t))\,S(\theta_i(t)), & t>0, \\[2pt]
        \theta_i(0)=\theta_i^0, &  i=1,\dots,N.
    \end{cases}
\end{equation}
Both systems \eqref{InertWinfree} and \eqref{GenWinfree} admit unique
global solutions by Cauchy--Lipschitz theory.

In the prototypical case
\begin{equation}\label{standard}
    S(\theta)=-\sin\theta,\qquad I(\theta)=1+\cos\theta,
\end{equation}
\eqref{InertWinfree} becomes
\begin{equation}\label{Winfree}
    \begin{cases}
        m\ddot\theta_i + \dot{\theta}_i=\nu_i-\dfrac\kappa N \sum_{j=1}^N (1+\cos\theta_j(t))\sin\theta_i(t), & t>0, \\[2pt]
        (\theta_i,\dot\theta_i)\big\rvert_{t=0^+} =(\theta_i^0,\omega_i^0), & i=1,\dots,N.
    \end{cases}
\end{equation}
We restrict to the model \eqref{Winfree} throughout this paper. Define the
\emph{order parameter}
\begin{equation}\label{order_parameter}
R\coloneqq\frac1N\sum_{j=1}^N(1+\cos\theta_j)\in[0,2],
\end{equation}
so that \eqref{Winfree} takes the mean-field form
\begin{equation}\label{Winfree_orderparam}
    m\ddot\theta_i+\dot{\theta}_i=\nu_i-\kappa R(t)\sin\theta_i,\qquad i\in[N].
\end{equation}

\paragraph{{\bf Relation to prior work.}}
We organize some relevant prior work along two axes (the list is not exhaustive): (i) \emph{first vs.\ second order}, that is, with or without
inertia, and (ii) \emph{Kuramoto vs.\ Winfree}, that is, with pairwise
sinusoidal coupling $\sin(\theta_j-\theta_i)$ versus separable coupling
$I(\theta_j)S(\theta_i)$. The corresponding
$2\times 2$ landscape is as follows.

\medskip
\begin{center}
\renewcommand{\arraystretch}{1.6}
\begin{tabular}{@{}p{0.15\textwidth}|p{0.4\textwidth}|p{0.4\textwidth}@{}}
 & \textbf{Kuramoto} & \textbf{Winfree} \\ \hline
\textbf{First-order}
 & Asymptotic phase-locking under small frequency spread and generic initial data~\cite{ha2016emergence,ha2020asymptotic}.
 & Oscillator death via a gradient-flow/\L ojasiewicz argument, order parameter bootstrapping, and a volumetric argument; the critical coupling strength was also computed~\cite{ryoo2026oscillator}. \\ \hline
\textbf{Second-order (inertial)}
 & Phase-locking from restricted initial data via Lyapunov-functional methods~\cite{choi2014complete}; from generic initial data with explicit $R_0^2$-type scaling~\cite{cho2025quantitative}; quantitative higher-order Tikhonov theorem~\cite{cho2025inertia}.
 & Linear stability of equilibria~\cite{ha2021emergent}; oscillator locking\footnote{This is a notion different from oscillator death; it roughly means that the oscillators rotate with a nonzero common velocity. See \cite[Section 3.1]{ryoo2026oscillator} for a detailed discussion.} under restrictive initial conditions~\cite{kang2022emergence}. \textbf{This paper} fills the remaining gap by proving oscillator death from generic initial data in the small intrinsic and initial frequency, small-inertia regime. \\
\end{tabular}
\end{center}
\medskip

\subsection{Main theorems}
Our first main result is a pathwise oscillator-death theorem with
explicit smallness thresholds, in the \emph{zero-inertia, large coupling limit}. This is a counterpart of the Kuramoto-result \cite[Theorem 1.1]{cho2025quantitative} for the inertial Winfree model.

\begin{theorem}[Pathwise synchronization]\label{thm:pathwise}
Fix absolute constants $a=\frac{1}{50}$, $b=\frac{1}{80}$, and $c=\frac{1}{20}$. For any initial data
$(\{\theta_i^0\}_{i=1}^N,\{\omega_i^0\}_{i=1}^N)$ and system parameters $(\{\nu_i\}_{i=1}^N,\kappa,m)$
satisfying $R_0:=R(0)>0$ and
\begin{equation}\label{eq:smallness-1}
\frac{\max_i |\nu_i|}{\kappa}<a\,R_0^{3/2},\qquad
m\kappa<b\,R_0^{3/2},\qquad
\frac{\max_i |\omega_i^0|}{\kappa}<c\,R_0^{3/2},
\end{equation}
the solution $\{\theta_i(t)\}_{i=1}^N$ to \eqref{Winfree} satisfies the following:
\begin{enumerate}
    \item (Oscillator death) For every $i\in[N]$, the limits
    $\theta_i^\infty:=\lim_{t\to\infty}\theta_i(t)$ and
    $\lim_{t\to\infty}\dot\theta_i(t)=0$ exist.
    \item (Lower bound on the order parameter) We have
    \[
    \inf_{t\ge 0}R(t)\ge\frac{R_0}{4}.
    \]
\end{enumerate}
\end{theorem}

\begin{remark}\,
\begin{enumerate}
	\item The constants $a,b,c$ are chosen to satisfy \eqref{eq:quant-check-1-abc} and \eqref{eq:quant-check-2-abc} below with comfortable margin. We do not claim their sharpness.
\item The exponent $3/2$ on $R_0$ is the analogue, for Winfree, of the
exponent $2$ that appears in the inertial Kuramoto pathwise theorem
of~\cite[Theorem 1.1]{cho2025quantitative}. It arises from the trapping mechanism in the argument of \S\ref{sec:pathwise-proof} (see in particular \eqref{eq:balance} and Remark~\ref{rmk:R032}).
\item For any $\varepsilon>0$,  one may improve statement (2) into $\inf_{t\ge 0}R(t)\ge (1-\varepsilon)R_0$ by taking the constants $a,b,c>0$ small enough depending on $\varepsilon$.
\item The parameters $\kappa, \nu_i, \omega_i^0$ have units of $\frac{1}{\mathrm{time}}$, whereas $m$ has units of $\mathrm{time}$, so \eqref{eq:smallness-1} is a dimensionless condition (and thus respects the time-dilatation symmetry of \cite[(2.7)]{cho2025quantitative}).
\end{enumerate}
\end{remark}

Our second main result is a qualitative synchronization result from generic initial data with better control on the limiting behavior of the order parameter. This is a counterpart of \cite[Theorem 1.2]{cho2025inertia}.

\begin{theorem}[Qualitative zero-inertia synchronization]\label{thm:qualitative_mto0}
Let $\{\theta_i^0\}_{i=1}^N\in\R^N$ satisfy $\theta_i^0\in (-\pi,\pi)$
for every $i\in[N]$, and let $\varepsilon>0$. Then
there exist positive numbers $a,b,c$, depending only on $\max_{i}|\theta_i^0|$ and
$\varepsilon$, such that if the initial velocities $\{\omega_i^0\}_{i=1}^N$ and
system parameters $(\{\nu_i\}_{i=1}^N,\kappa,m)$ satisfy
\begin{equation}\label{eq:smallness-2}
\frac{\max_i |\nu_i|}{\kappa}<a,\qquad
m\kappa<b,\qquad \frac{\max_i |\omega^0_i|}{\kappa}<c,
\end{equation}
then oscillator death occurs for the solution $\Theta(t)$ of
\eqref{Winfree} and
\[
\lim_{t\to\infty}R(t)>2-\varepsilon.
\]
\end{theorem}

The qualitative statement of Theorem~\ref{thm:qualitative_mto0}
complements the pathwise statement of
Theorem~\ref{thm:pathwise}: Theorem~\ref{thm:pathwise} gives an
\emph{explicit} condition on the parameters that depends polynomially
on $R_0$, but only recovers $\inf R\ge R_0/4$, while
Theorem~\ref{thm:qualitative_mto0} delivers the stronger conclusion
$\lim_{t\to\infty} R(t)>2-\varepsilon$ at the cost of a possibly worse dependence on
$(\{\theta_i^0\}_{i=1}^N,\varepsilon)$. The proof of
Theorem~\ref{thm:qualitative_mto0} combines a quantitative Tikhonov
theorem (Proposition~\ref{prop:mto0quant}) with a first-order Winfree
input from~\cite{ryoo2026oscillator}; see \S\ref{sec:qualitative}.

\subsection{Notation and conventions}\label{sec:notation}
We denote $[N]=\{1,\dots,N\}$. We use capital Greek letters for the
$N$-tuples of the corresponding lower Greek letters, with $\mathcal V$
for the natural frequencies:
\[
\Theta\coloneqq(\theta_1,\dots,\theta_N),\qquad
\Omega\coloneqq(\omega_1,\dots,\omega_N),\qquad
\mathcal V\coloneqq(\nu_1,\dots,\nu_N).
\]
The symbol $\|\cdot\|_\infty$ denotes the $\ell^\infty$-norm on $\R^N$,
\[
\|\Theta\|_\infty\coloneqq\max_{i\in[N]}|\theta_i|,\qquad
\|\Omega\|_\infty\coloneqq\max_{i\in[N]}|\omega_i|,\qquad
\|\mathcal V\|_\infty\coloneqq\max_{i\in[N]}|\nu_i|.
\]
Given a subset $\mathcal{B}\subset [N]$, we write
\[
\Theta_\mathcal{B}\coloneqq(\theta_i)_{i\in \mathcal{B}},\qquad
\Omega_\mathcal{B}\coloneqq(\omega_i)_{i\in \mathcal{B}},\qquad
\mathcal V_\mathcal{B}\coloneqq(\nu_i)_{i \in \mathcal{B}},
\]
so that
\[
\|\Theta_\mathcal{B}\|_\infty\coloneqq\max_{i\in\mathcal{B}}|\theta_i|,\qquad
\|\Omega_\mathcal{B}\|_\infty\coloneqq\max_{i\in \mathcal{B}}|\omega_i|,\qquad
\|\mathcal V_\mathcal{B}\|_\infty\coloneqq\max_{i\in \mathcal{B}}|\nu_i|.
\]

We write $R_0:=R(\Theta^0)$ for the initial order parameter. For $\mathcal{B}\subset [N]$, we introduce the notation
\[
R_\mathcal{B}(\Theta)=\frac 1N\sum_{i\in \mathcal{B}}(1+\cos\theta_i),
\]
so that, for example, we have
\begin{equation}\label{eq:order-parameter-split}
R(\Theta) = R_{[N]}(\Theta)= R_\mathcal{B}(\Theta) + R_{[N]\setminus \mathcal{B}}(\Theta).
\end{equation}

Throughout this paper we write $\Theta(m,t)$ for the solution of
\eqref{Winfree} with inertia $m\ge0$ and initial position and velocity data $\Theta^0$ and $\Omega^0$, and $\Theta(0,t)$ for the solution
of \eqref{GenWinfree} with \eqref{standard} and initial position data $\Theta^0$.  If there is no confusion, $\Theta(t)$ will mean $\Theta(m,t)$. We write
$F\colon\R^N\to\R^N$ for the Winfree right-hand side,
\begin{equation}\label{def:F}
F_i(\Theta)\coloneqq\nu_i-\kappa R(\Theta)\sin\theta_i,\qquad i\in[N],
\end{equation}
so that \eqref{Winfree_orderparam} reads
$m\ddot\Theta+\dot\Theta=F(\Theta)$ and \eqref{GenWinfree}--\eqref{standard}
reads $\dot\Theta=F(\Theta)$. We set
\begin{equation}\label{def:MFLF}
M_F\coloneqq\|\mathcal V\|_\infty+2\kappa,
\end{equation}
which is an
$\ell^\infty$-bound on $F$ since $R\in [0,2]$. For $\mathcal{B}\subset [N]$, we set
\begin{equation}\label{def:MFB}
M_{F,\mathcal{B}} \coloneqq \|\mathcal{V}_\mathcal{B}\|_\infty +2\kappa,
\end{equation}
which is an $\ell^\infty$-bound on $F_i$ for $i\in \mathcal{B}$.

\subsection{Organization}
Section~\ref{sec:prelim} collects preliminary facts: the inertial
gradient-flow structure and the \L ojasiewicz theorem
(Proposition~\ref{prop:lojasiewicz}), Duhamel's principle and the resulting speed-control lemmas (Lemmas \ref{lem:speedlim} and \ref{lem:orderlowering}), and the initial layer control (Lemma \ref{lem:initlayer}).
Section~\ref{sec:tikhonov} develops the quantitative higher-order
Tikhonov theorem
(Proposition~\ref{prop:mto0quant}). Section~\ref{sec:pathwise-proof}
proves Theorem~\ref{thm:pathwise}: we establish a single-oscillator trapping
lemma (Lemma~\ref{lem:single}) and the main partial oscillator death
criterion (Proposition~\ref{prop:pod}), and assemble these into the
pathwise Theorem \ref{thm:pathwise} via an initial-layer / condensation / persistence
argument. Section~\ref{sec:qualitative} proves
Theorem~\ref{thm:qualitative_mto0} by using
a quantitative Tikhonov theorem (Proposition~\ref{prop:mto0quant}) and the first-order Winfree theorem
of~\cite{ha2015emergence}. Section~\ref{sec:volume} discusses possible volumetric arguments and the existence of a Lyapunov functional. Section~\ref{sec:conclusion} offers
concluding remarks and open problems.

\section{Preliminaries}\label{sec:prelim}

\subsection{Inertial gradient-flow structure and \L ojasiewicz theorem}

The inertial Winfree model~\eqref{Winfree} possesses a real-analytic
inertial gradient-flow formulation. Define the real-analytic potential
$P\colon\R^N\to\R$ by
\begin{equation}\label{def:P}
P(\Theta)\coloneqq-\sum_{k=1}^N\nu_k\theta_k-\frac{\kappa N}{2}R(\Theta)^2
=-\sum_{k=1}^N\nu_k\theta_k-\frac{\kappa}{2N}\Bigl(\sum_{k=1}^N(1+\cos\theta_k)\Bigr)^{\!2}.
\end{equation}
A direct computation yields
\[
\partial_{\theta_i}P
=-\nu_i-\kappa N R\cdot\Bigl(-\tfrac{1}{N}\sin\theta_i\Bigr)
=-\nu_i+\kappa R(\Theta)\sin\theta_i
=-F_i(\Theta),
\]
so that \eqref{Winfree_orderparam} takes the inertial gradient-flow form
\begin{equation}\label{eq:gradflow}
m\ddot\Theta+\dot\Theta\ =\ -\nabla_\Theta P(\Theta).
\end{equation}
We denote the set of critical
points of the real-analytic potential $P$ by
\[
\mathcal S\coloneqq\{(\Theta,{\mathbf 0})\in\R^{2N}\colon \nabla_\Theta P(\Theta)=0\},
\]
whose projection to the first factor is precisely the set of
\emph{death states}, i.e., equilibria of the first-order Winfree model
\eqref{GenWinfree}--\eqref{standard}.

The \L ojasiewicz gradient theorem (\cite{lojasiewicz1963propriete,
lojasiewicz1982trajectoires}) asserts that a bounded solution to a
first-order gradient flow of a real-analytic potential converges to a
critical point of the potential. This was extended to inertial
gradient flows in \cite[Theorem~1.1]{haraux1998convergence}
and~\cite[Corollary~5.1]{begout2015damped}, and was observed
in~\cite[Proposition~2.1]{choi2014complete} to apply to the inertial
Kuramoto model.

\begin{proposition}[\cite{haraux1998convergence,begout2015damped}]\label{prop:lojasiewicz}
Let $(\Theta,\Omega)$ be a global solution to~\eqref{Winfree} satisfying
the a priori uniform bound
\[
\|\Theta\|_{W^{1,\infty}}\coloneqq\|\Theta\|_{L^\infty[0,\infty)}+\|\Omega\|_{L^\infty[0,\infty)}<\infty.
\]
Then there exists $(\Theta^\infty,{\mathbf 0})\in\mathcal S$ such that
\[
\lim_{t\to\infty}\bigl(\|\Theta(t)-\Theta^\infty\|_\infty+\|\Omega(t)\|_\infty\bigr)=0.
\]
Moreover, the convergence is algebraic: there exist constants $c,C>0$ with
$\|\Theta(t)-\Theta^\infty\|_\infty\le C(1+t)^{-c}$ for all $t\ge 0$.
\end{proposition}

\begin{remark}\label{rmk:velocity-bound}
By Lemma~\ref{lem:speedlim} below, $\|\Omega\|_{L^\infty[0,\infty)}<\infty$
automatically, so in order to apply
Proposition~\ref{prop:lojasiewicz} it suffices to verify
$\|\Theta\|_{L^\infty[0,\infty)}<\infty$. This is satisfied, e.g., if we know that each $\theta_i$
stays in a compact proper arc for large enough time.
\end{remark}

\begin{remark}
The potential~\eqref{def:P} only differs from the pairwise-coupling
potential
$P^K(\Theta)=-\sum_k\nu_k\theta_k+\tfrac{\kappa}{2N}\sum_{k,\ell}(1-\cos(\theta_k-\theta_\ell))$
used for inertial Kuramoto in~\cite{choi2014complete}. The application of \cite{haraux1998convergence,begout2015damped} is the same for the Winfree and Kuramoto systems.
\end{remark}
\begin{remark}
The Winfree potential is unbounded below
due to the linear drift $-\sum\nu_k\theta_k$, and $\mathbb{R}^N$ is noncompact. That is why we need boundedness of the trajectory (along with real-analyticity of $P$) in Proposition~\ref{prop:lojasiewicz}.

\end{remark}

\subsection{Duhamel's principle and the speed-limit lemma}\label{sec:duhamel}

We can rewrite~\eqref{Winfree_orderparam} as the $2N$-dimensional first-order
system
\begin{equation}\label{standardduhamel}
\begin{cases}
\dot\theta_i(t)=\omega_i(t),\\[2pt]
\dot\omega_i(t)=\dfrac{1}{m}\bigl(F_i(\Theta(t))-\omega_i(t)\bigr),
\end{cases}\qquad (\theta_i,\omega_i)\big|_{t=0^+}=(\theta_i^0,\omega_i^0),\quad i\in[N].
\end{equation}
Viewing the equation for $\omega_i$ as an inhomogeneous linear ODE in
$\omega_i$, Duhamel's principle gives
\begin{equation}\label{eq:duhamel}
\dot\theta_i(t)=\omega_i^0 e^{-t/m}+\frac{1}{m}\int_0^t e^{-(t-s)/m}F_i(\Theta(s))\,ds,
\end{equation}
and integrating once more, with Fubini,
\begin{equation}\label{eq:duhamel-theta}
\theta_i(t)=\theta_i^0+m\omega_i^0(1-e^{-t/m})+\int_0^t F_i(\Theta(s))\bigl(1-e^{-(t-s)/m}\bigr)\,ds.
\end{equation}

\begin{lemma}[Speed limit]\label{lem:speedlim}
For every solution to~\eqref{Winfree}, every $i\in[N]$, and every $t\ge 0$,
\[
|\dot\theta_i(t)|\le |\omega_i^0| e^{-t/m}+(|\nu_i|+2\kappa)(1-e^{-t/m})\le\|\Omega^0\|_\infty+M_F,
\]
and
\[
\Bigl|\frac{d}{dt}(R(t)\sin\theta_i(t))\Bigr|\le 2\|\Omega^0\|_\infty e^{-t/m}+2 M_F(1-e^{-t/m}).
\]

For $\mathcal{B}\subset [N]$ and $i\in \mathcal{B}$, we have
\[
\Bigl|\frac{d}{dt}(R_\mathcal{B}(t)\sin\theta_i(t))\Bigr|\le \frac{2|\mathcal{B}|}{N}\left(\|\Omega_\mathcal{B}^0\|_\infty e^{-t/m}+M_{F,\mathcal{B}}(1-e^{-t/m})\right).
\]
\end{lemma}
\begin{proof}
The first bound follows term by term from~\eqref{eq:duhamel} (recalling \eqref{def:F}),
$|F_i|\le |\nu_i|+2\kappa \le M_F$, and
$\int_0^t e^{-(t-s)/m}\,ds=m(1-e^{-t/m})$.

We prove the third bound first; the second then follows. For the third bound, let
$M_\mathcal{B}(t):=\|\Omega_\mathcal{B}^0\|_\infty e^{-t/m}+M_{F,\mathcal{B}}(1-e^{-t/m})$. Then for $i\in \mathcal{B}$, we have $|\dot\theta_i(t)|\le M_\mathcal{B}(t)$, 
$|\tfrac{d}{dt}\sin\theta_i|=|\cos\theta_i|\,|\dot\theta_i|\le M_\mathcal{B}(t)|\cos\theta_i|$ and
\begin{equation}\label{eq:R-deriv}
	|\dot R_\mathcal{B}(t)|\le\frac1N\sum_{j\in \mathcal{B}}|\sin\theta_j|\,|\dot\theta_j|\le M_\mathcal{B}(t)\cdot\frac1N\sum_{j\in \mathcal{B}}|\sin\theta_j|,
\end{equation}
so that
\begin{align*}
	\Bigl|\frac{d}{dt}\bigl(R_\mathcal{B}(t)\sin\theta_i(t)\bigr)\Bigr|
	&\le R_\mathcal{B}(t)\,|\cos\theta_i(t)|\,|\dot\theta_i(t)|+|\sin\theta_i(t)|\,|\dot R_\mathcal{B}(t)|\\
	&\le M_\mathcal{B}(t)\Bigl(\frac1N\!\sum_{j\in \mathcal{B}}(1+\cos\theta_j)|\cos\theta_i|+\frac1N\!\sum_{j\in \mathcal{B}}|\sin\theta_j||\sin\theta_i|\Bigr)\\
	&\le\frac{M_\mathcal{B}(t)}{N}\sum_{j\in \mathcal{B}}\bigl(|\cos\theta_i|+|\cos\theta_j||\cos\theta_i|+|\sin\theta_j||\sin\theta_i|\bigr)\\
	&\le\frac{M_\mathcal{B}(t)}{N}\sum_{j\in \mathcal{B}} 2=\frac{2|\mathcal{B}|}{N}M_\mathcal{B}(t),
\end{align*}
using $|\cos\theta_j||\cos\theta_i|+|\sin\theta_j||\sin\theta_i|\le1$
by Cauchy--Schwarz.

The second follows from the third by setting $\mathcal{B}=[N]$.
\end{proof}

The next lemma upgrades the speed limit into a \emph{pointwise}
approximation of $\dot\theta_i(t)$ by its ``first-order-like'' part.

\begin{lemma}[Approximating second-order velocity by first-order]\label{lem:orderlowering}
For every solution to~\eqref{Winfree}, every $i\in[N]$, and every $t\ge 0$,
\begin{multline*}
\bigl|\dot\theta_i(t)-\omega_i^0 e^{-t/m}-\nu_i(1-e^{-t/m})+\kappa R(t)\sin\theta_i(t)(1-e^{-t/m})\bigr|\\
\le 2\kappa\,\|\Omega^0\|_\infty t e^{-t/m}(1-e^{-t/m})+2m\kappa M_F(1-e^{-t/m})^{3}.
\end{multline*}

If $\mathcal{B}\subset [N]$ and $i\in \mathcal{B}$, then
\begin{multline*}
	\bigl|\dot\theta_i(t)-\omega_i^0 e^{-t/m}-\nu_i(1-e^{-t/m})+\kappa R_\mathcal{B}(t)\sin\theta_i(t)(1-e^{-t/m})\bigr|\\
	\le 2\kappa (1-e^{-t/m})\left(\frac{|\mathcal{B}|}{N} \left(\|\Omega_\mathcal{B}^0\|_\infty te^{-t/m}+m M_{F,\mathcal{B}}(1-e^{-t/m})^2\right)+\frac{N-|\mathcal{B}|}{N} \right).
\end{multline*}
\end{lemma}
\begin{proof}
	The first inequality follows from the second by setting $\mathcal{B}=[N]$. We prove the second inequality. By~\eqref{eq:duhamel}, we have the decomposition
\begin{align*}
\dot\theta_i(t)&=\omega_i^0 e^{-t/m}+\nu_i(1-e^{-t/m})-\frac{\kappa}{m}\int_0^t e^{-(t-s)/m}R(s)\sin\theta_i(s)\,ds\\
&\stackrel{\mathclap{\eqref{eq:order-parameter-split}}}{=}\omega_i^0 e^{-t/m}+\nu_i(1-e^{-t/m})-\frac{\kappa}{m}\int_0^t e^{-(t-s)/m}R_\mathcal{B}(s)\sin\theta_i(s)ds\\
&\qquad-\frac{\kappa}{m}\int_0^t e^{-(t-s)/m}R_{[N]\setminus\mathcal{B}}(s)\sin\theta_i(s)\,ds\\
&=\omega_i^0 e^{-t/m}+\nu_i(1-e^{-t/m}) - \kappa R_\mathcal{B}(t)\sin\theta_i(t)(1-e^{-t/m})\\
&\quad +\frac{\kappa}{m}\int_0^t e^{-(t-s)/m}(R_\mathcal{B}(t)\sin\theta_i(t)-R_\mathcal{B}(s)\sin\theta_i(s))ds\\
& \quad -\frac{\kappa}{m}\int_0^t e^{-(t-s)/m}R_{[N]\setminus\mathcal{B}}(s)\sin\theta_i(s)\,ds
\end{align*}
so that we may bound
\begin{align*}
	&\Bigl|\dot\theta_i(t)-\omega_i^0 e^{-t/m}-\nu_i(1-e^{-t/m})+\kappa R_\mathcal{B}(t)\sin\theta_i(t)(1-e^{-t/m})\Bigr|\\
	&\qquad\le \Bigl|\frac{\kappa}{m}\int_0^t\bigl(R_\mathcal{B}(t)\sin\theta_i(t)-R_\mathcal{B}(s)\sin\theta_i(s)\bigr)e^{-(t-s)/m}\,ds\Bigr|\\
	&\qquad\qquad +\left|\frac{\kappa}{m}\int_0^t e^{-(t-s)/m}R_{[N]\setminus\mathcal{B}}(s)\sin\theta_i(s)ds\right|\\
	&\qquad\le\frac{\kappa}{m}\int_0^t|R_\mathcal{B}(t)\sin\theta_i(t)-R_\mathcal{B}(s)\sin\theta_i(s)|e^{-(t-s)/m}\,ds\\
		&\qquad\qquad +\frac{N-|\mathcal{B}|}{N}2\kappa (1-e^{-t/m})
\end{align*}
where in the second inequality we used the triangle inequality, $R_{[N]\setminus\mathcal{B}}\le 2\frac{N-|\mathcal{B}|}{N}$, and $|\sin|\le 1$.

By the third bound of Lemma~\ref{lem:speedlim},
\begin{align*}
	|R_\mathcal{B}(t)\sin\theta_i(t)-R_\mathcal{B}(s)\sin\theta_i(s)|
	&\le 2\frac{|\mathcal{B}|}{N}\int_s^t\bigl(\|\Omega_\mathcal{B}^0\|_\infty e^{-\tau/m}+M_{F,\mathcal{B}}(1-e^{-\tau/m})\bigr)d\tau\\
	&=2\frac{|\mathcal{B}|}{N}\left(m\|\Omega_\mathcal{B}^0\|_\infty(e^{-s/m}-e^{-t/m})+M_{F,\mathcal{B}}(t-s-me^{-s/m}+me^{-t/m})\right).
\end{align*}
Substituting, integrating, and using the elementary inequalities
\[
x e^{-x}-e^{-x}+e^{-2x}\le x e^{-x}(1-e^{-x}),\quad
1-2xe^{-x}-e^{-2x}\le(1-e^{-x})^{3},\quad x\ge 0,
\]
(with $x=t/m$), we obtain
\begin{align*}
&\frac{\kappa}{m}\int_0^t|R_\mathcal{B}(t)\sin\theta_i(t)-R_\mathcal{B}(s)\sin\theta_i(s)|e^{-(t-s)/m}\,ds\\
	&\le 2\frac{|\mathcal{B}|}{N}\kappa \left(\|\Omega_\mathcal{B}^0\|_\infty e^{-t/m}\int_0^t \left(1-e^{-(t-s)/m}\right)ds+\frac{ M_{F,\mathcal{B}}}{m}\int_0^t(t-s-me^{-s/m}+me^{-t/m})e^{-(t-s)/m}ds\right)\\
	&=2\frac{|\mathcal{B}|}{N}\kappa\left(\|\Omega_\mathcal{B}^0\|_\infty e^{-t/m}\left(t-m+me^{-t/m}\right)+ M_{F,\mathcal{B}}\left(m-2te^{-t/m}-me^{-2t/m}\right)\right)\\
	&\le 2\frac{|\mathcal{B}|}{N}\kappa \left(\|\Omega_\mathcal{B}^0\|_\infty te^{-t/m}(1-e^{-t/m})+m M_{F,\mathcal{B}}(1-e^{-t/m})^3\right).
\end{align*}
Combining estimates, we get the stated result.
\end{proof}

\begin{lemma}[Initial layer effect on $R$]\label{lem:initlayer}
If $\delta\in(0,1)$ and $\eta>0$ satisfy
\[
m\|\Omega^0\|_\infty(1-e^{-\eta})+M_F\cdot m(\eta-1+e^{-\eta})<(1-\delta)R_0,
\]
then $R(t)>\delta R_0$ for $t\in [0,\eta m]$.
\end{lemma}
\begin{proof}
By~\eqref{eq:R-deriv} in the proof of Lemma~\ref{lem:speedlim},
$|\dot R(t)|\le M(t)=\|\Omega^0\|_\infty e^{-t/m}+M_F(1-e^{-t/m})$,
so for $t\in [0,\eta m]$,
\begin{align*}
R(t)
&\ge R_0-\int_0^{t}M(s)\,ds\\
&\ge R_0-\int_0^{\eta m}M(s)\,ds\\
&=R_0-m\|\Omega^0\|_\infty(1-e^{-\eta})-M_F\cdot m(\eta-1+e^{-\eta})\\
&>\delta R_0. \qedhere
\end{align*}
\end{proof}

\section{Quantitative higher-order Tikhonov theorem}\label{sec:tikhonov}

We mentioned earlier in the introduction that solutions to \eqref{GenWinfree} with \eqref{standard} are formally obtained from those of \eqref{InertWinfree} when passing to the limit $m\to 0$. 
Tikhonov's classical theorem~\cite{tikhonov1952systems,vasil1963asymptotic} makes this precise: it asserts that, for fixed initial data, the solution $\Theta(m,t)$ of the
inertial Winfree problem converges, as $m\to 0$, to the solution
$\Theta(0,t)$ of the first-order Winfree problem~\eqref{GenWinfree}
with~\eqref{standard}, in the Fr\'echet topologies $C^0[0,\infty)$ and $C^1(0,\infty)$. The classical theorem is qualitative. In this section we
derive a \emph{quantitative} and \emph{higher-order} version, namely
Proposition~\ref{prop:mto0quant} below, giving explicit rates of
convergence in the topologies $C^0[0,\infty)$ and $C^\infty(0,\infty)$.
This is the Winfree analogue of~\cite[Proposition~3.2]{cho2025inertia}.

\subsection{Qualitative Tikhonov theorem}
Let us first restate the classical theorem for the reader's convenience.
\begin{proposition}[Classical Tikhonov theorem]\label{prop:mto0}
Fix initial data $(\Theta^0,\Omega^0)$, intrinsic velocities $\mathcal V$,
and a coupling strength $\kappa>0$. For each $m>0$, let $\Theta(m,t)$
denote the solution to~\eqref{Winfree} and $\Theta(0,t)$ the solution
to~\eqref{GenWinfree}--\eqref{standard}. Then for every $i\in[N]$,
\[
\lim_{m\to 0}\sup_{t\in[0,T]}|\theta_i(m,t)-\theta_i(0,t)|=0\quad(\forall T>0),
\]
and
\[
\lim_{m\to 0}\sup_{t\in[T_1,T_2]}|\omega_i(m,t)-\omega_i(0,t)|=0\quad(\forall T_2>T_1>0).
\]
\end{proposition}
\begin{proof}
System~\eqref{Winfree} in the form~\eqref{standardduhamel} fits the
standard hypotheses of~\cite[Theorem~1.1]{vasil1963asymptotic}: the
``fast'' variable $\omega$ has a unique globally stable equilibrium
$\omega=F(\Theta)$ for each frozen $\Theta$, and the ``slow'' variable
$\Theta$ evolves via $\dot\Theta=\omega$. The conclusion follows.
\end{proof}

\begin{remark}
Proposition~\ref{prop:mto0} cannot give uniform-in-time bounds
(convergence in $L^\infty[0,\infty)$) because it is not true. Indeed, even for $N=1$, the two
systems $\dot\theta=1-\tfrac12(1+\cos\theta)\sin\theta$ starting at $\theta^0=0$ and
$\tfrac12\ddot\theta+\dot\theta=1-\tfrac12(1+\cos\theta)\sin\theta$
starting at $\theta^0=\dot\theta^0=0$ have phase difference growing
linearly in $t$. Nor does Proposition~\ref{prop:mto0} give convergence in $C^1$ on a
neighborhood of $t=0$, because again it is not true: the inertial system prescribes the arbitrary
value $\dot\theta_i(0)=\omega_i^0$, while the first-order system
mandates $\dot\theta_i(0)=F_i(\Theta^0)$. Any approximation of the
second-order by the first-order model in the $C^1$-topology, such as Lemma \ref{lem:orderlowering}, must
therefore be carried out after an initial time layer of the form
$[0,\eta m]$.
\end{remark}

The drawback of Tikhonov's theorem is that it does not give the quantitative bounds needed in this paper. By working with the ODEs of the Winfree model directly, we obtain explicit bounds as below in Proposition \ref{prop:mto0quant}, where we also show $C^\infty(0,\infty)$ convergence. One could verify these bounds directly by adapting the Gronwall arguments, but to streamline the exposition we import the results of \cite{cho2025inertia} for the Kuramoto model via an embedding of the Winfree system into the Kuramoto system.

\subsection{Embedding the Winfree model into the Kuramoto model}
\label{subsec:embedding}

The main observation of this section is that the first-order/inertial Winfree
system of $N$ oscillators is exactly an invariant submanifold of a
first-order/inertial Kuramoto system of $4N$ oscillators, obtained by adjoining
phase-reflected copies and a cluster of spectator oscillators frozen
at the origin. We are not aware of prior work that observes this fact. This reduces the quantitative Tikhonov problem for
Winfree to the one already treated in~\cite{cho2025inertia} for
Kuramoto, at the negligible cost of a constant factor in the rate.

\begin{lemma}[Winfree-to-Kuramoto embedding]\label{lem:embedding}
Let $\Theta(t)=(\theta_1,\dots,\theta_N)$ solve~\eqref{Winfree}
with natural frequencies $\mathcal V=(\nu_1,\dots,\nu_N)$, coupling
$\kappa$, inertia $m$, and initial data $(\Theta^0,\Omega^0)$. Define
$\Phi(t)=(\varphi_1,\dots,\varphi_{4N})$ by
\[
\varphi_i(t)=\theta_i(t),\quad \varphi_{N+i}(t)=-\theta_i(t),\quad
\varphi_{2N+k}(t)\equiv 0,\qquad i\in[N],\ k\in[2N],
\]
with natural frequencies
$\widetilde{\mathcal V}=(\nu_1,\dots,\nu_N,-\nu_1,\dots,-\nu_N,0,\dots,0)$,
coupling strength $\tilde\kappa=2\kappa$, inertia $m$, and initial
data $(\Phi^0,\dot\Phi^0)$ inherited from the embedding. Then $\Phi$
solves the inertial Kuramoto system
\begin{equation}\label{eq:Kur-embed}
m\ddot\varphi_k+\dot\varphi_k=\tilde\nu_k+\frac{\tilde\kappa}{4N}\sum_{l=1}^{4N}\sin(\varphi_l-\varphi_k),\qquad k\in[4N].
\end{equation}

Likewise, if $\Theta(t)$ solves the first-order Winfree system \eqref{GenWinfree} under \eqref{standard} with natural frequencies $\mathcal V=(\nu_1,\dots,\nu_N)$, coupling
$\kappa$, and initial data $\Theta^0$, then $\Phi$ defined as above with natural frequencies
$\widetilde{\mathcal V}=(\nu_1,\dots,\nu_N,-\nu_1,\dots,-\nu_N,0,\dots,0)$,
coupling strength $\tilde\kappa=2\kappa$, and initial
data $\Phi^0$ inherited from the embedding solves the first-order Kuramoto system
\begin{equation}\label{eq:Kur-first}
\dot\varphi_k=\tilde\nu_k+\frac{\tilde\kappa}{4N}\sum_{l=1}^{4N}\sin(\varphi_l-\varphi_k),\qquad k\in[4N].
\end{equation}
\end{lemma}
\begin{proof}
We check each block separately. For $k=i\in[N]$, the Kuramoto right-hand
side equals
\begin{align*}
\nu_i+\frac{2\kappa}{4N}\sum_{l=1}^{4N}\sin(\theta_l-\varphi_i)
&=\nu_i-\frac{\kappa}{2N}\Bigl[\sum_j\sin(\theta_i-\theta_j)+\sum_j\sin(\theta_i+\theta_j)+2N\sin\theta_i\Bigr]\\
&=\nu_i-\frac{\kappa}{2N}\Bigl[2\sin\theta_i\sum_j\cos\theta_j+2N\sin\theta_i\Bigr]\\
&=\nu_i-\frac{\kappa}{N}\sin\theta_i\sum_{j=1}^N(1+\cos\theta_j),
\end{align*}
which matches the Winfree right-hand side. For
$k=N+i$ the computation is the same with $\theta_i\mapsto-\theta_i$
and $\nu_i\mapsto-\nu_i$, so the ansatz $\varphi_{N+i}=-\theta_i$ is
preserved. For $k=2N+j$, one checks that both the intrinsic frequency
and the coupling sum vanish by odd symmetry, so
$\varphi_{2N+j}\equiv 0$ is an invariant solution.
\end{proof}

\begin{remark}\label{rmk:diam-sup}
Because the embedded system contains $2N$ oscillators pinned at $0$,
the sup-norm $\|\Theta\|_\infty$ of the Winfree configuration equals
half the \emph{phase diameter} $\max_{k,l}|\varphi_k-\varphi_l|$ of the
embedded Kuramoto configuration. Hence any
diameter-based (equivalently, Galilean-invariant) norm bound for the
Kuramoto system \eqref{eq:Kur-embed} translates verbatim into a
sup-norm bound for the Winfree system~\eqref{Winfree}. The same holds
for velocities since $\dot\varphi_{2N+j}\equiv 0$.
\end{remark}

\subsection{The quantitative Tikhonov theorem}

We present a quantitative and higher-order version of Tikhonov's theorem, namely a $C[0,\infty)\cap C^\infty(0,\infty)$ convergence statement with explicit bounds. Using this, we will prove Theorem \ref{thm:qualitative_mto0} later in \S\ref{sec:qualitative} by a comparison argument.

\begin{proposition}[Quantitative higher-order Tikhonov theorem]\label{prop:mto0quant}
Fix initial data $(\Theta^0,\Omega^0)$, intrinsic velocities $\mathcal V$,
and a coupling strength $\kappa>0$. For each $m>0$, let $\Theta(m,t)$ denote the solution to \eqref{Winfree} with initial position $\Theta^0$ and initial velocity $\Omega^0$, and let $\Theta(0,t)$ denote the solution to \eqref{GenWinfree}--\eqref{standard}, with the same
initial position $\Theta^0$. Then:
\begin{enumerate}
\item \emph{(Phase convergence, $C^0[0,\infty)$)} For all $t\ge 0$,
\begin{equation}\label{eq:tik-phase}
\|\Theta(m,t)-\Theta(0,t)\|_\infty < \frac 12 m \Big(\|\Omega^0-\mathcal{V}\|_\infty+2\kappa \Big )e^{4\kappa t}.
\end{equation}

\item \emph{(Velocity convergence, $C^0(0,\infty)$)} For all $t\ge 0$,
\begin{multline}\label{eq:tik-vel}
	\|\dot\Theta(m,t)-\dot\Theta(0,t)\|_\infty<\frac 12\Big(\|\Omega^0-\mathcal{V}\|_\infty+2\kappa \Big)e^{-t/m}+ 2m\kappa(\|\mathcal{V}\|_\infty+2\kappa) +2m\kappa \Big (\|\Omega^0-\mathcal{V}\|_\infty+2\kappa\Big )e^{4\kappa t}.
\end{multline}

\item \emph{(Higher derivatives, $C^\infty(0,\infty)$)} For every
integer $n\ge 1$ and every $t\ge 0$,
\begin{equation}\label{eq:tik-higher}
\|\Theta^{(n)}(m,t)-\Theta^{(n)}(0,t)\|_\infty
\le C_n \left[m^{-n}\Bigl(1+\tfrac{t}{m}\Bigr)^{\!n}e^{-t/m}+m^{1-n}\kappa \Bigl(1+\tfrac{t}{m}\Bigr)^{\!n} e^{4\kappa t-t/m}\right]+C_n'm\kappa^{n+1} e^{4\kappa t}
\end{equation}
where $C_n$ depends polynomially on $m\kappa$,
$m\|\mathcal V\|_\infty$, $m\|\Omega^0\|_\infty$, and $C_n'$ depends polynomially on $\frac{\|\Omega^0\|_\infty}{\kappa},\frac{\|\mathcal{V}\|_\infty}\kappa$ (with the polynomial dependence depending on $n$).
\end{enumerate}
\end{proposition}

\begin{proof}
	Lemma~\ref{lem:embedding} reduces the proof to a direct application
	of~\cite[Proposition~3.2]{cho2025inertia}. We spell out the reduction.

	Fix the Winfree initial data $(\Theta^0,\Omega^0)$ and let
	$(\Phi^0,\dot\Phi^0)\in\R^{4N}\times\R^{4N}$ be the embedded Kuramoto
	initial data defined by Lemma~\ref{lem:embedding}. Write $\Phi(m,t)$
	and $\Phi(0,t)$ for the solutions of the second-order Kuramoto
	system~\eqref{eq:Kur-embed} and \eqref{eq:Kur-first} respectively, both starting from
	$\Phi^0$ (for the first-order system, use $\Phi^0$ as the initial
	phase; for the second-order system, use $(\Phi^0,\dot\Phi^0)$), with $\tilde\kappa = 2\kappa$.
	By Lemma~\ref{lem:embedding}, for both $m>0$ and $m=0$ the components of
	$\Phi$ satisfy $\varphi_i=\theta_i$, $\varphi_{N+i}=-\theta_i$, and
	$\varphi_{2N+k}\equiv 0$ for all $t\ge 0$. In particular, similar to
	Remark~\ref{rmk:diam-sup},
	\begin{equation}\label{eq:equal-norms}
		\|\Theta^{(n)}(m,t)-\Theta^{(n)}(0,t)\|_\infty=\frac 12\|\Phi^{(n)}(m,t)-\Phi^{(n)}(0,t)\|_{\infty}.
	\end{equation}
	
	By \cite[Proposition 3.2]{cho2025inertia}, we have (denoting $\mathcal{D}(X)\coloneqq \max_{i,j\in [4N]}|x_i-x_j|$ for $X\in \mathbb{R}^{4N}$):
\begin{enumerate}
	\item (Convergence of $\Phi(m, \cdot)$):
	\begin{align*}
			& \|\Phi(m,t)-\Phi(0,t)\|_\infty \\
			&\hspace{1cm} < m\left|\frac{\max_i(\tilde{\omega}_i^0-\tilde{\nu}_i)+\min_i(\tilde{\omega}_i^0-\tilde{\nu}_i)}2\right|  +\frac 12 m \Big(\mathcal{D} (\tilde{\Omega}^0-\tilde{\mathcal{V}})+2\tilde{\kappa} \Big )e^{2\tilde{\kappa} t}\\
			&\hspace{1cm}=0+m \Big(\|\Omega^0-\mathcal{V}\|_\infty+2\kappa \Big )e^{4\kappa t}
	\end{align*}
	so that by \eqref{eq:equal-norms},
	\[
	\|\Theta(m,t)-\Theta(0,t)\|_\infty < \frac 12 m \Big(\|\Omega^0-\mathcal{V}\|_\infty+2\kappa \Big )e^{4\kappa t}.
	\]
	\item 
	(Convergence of ${\dot \Phi}(m, \cdot)$):
	\begin{align*}
			&\|\dot\Phi(m,t)-\dot\Phi(0,t)\|_\infty \\
			&\hspace{1cm} < \Big(\|\tilde{\Omega}^0-\tilde{\mathcal{V}}\|_\infty+\tilde{\kappa} \Big)e^{-t/m}+ m\tilde{\kappa}(\mathcal{D}(\tilde{\mathcal{V}})+2\tilde{\kappa}) +m\tilde{\kappa} \Big (\mathcal{D}(\tilde{\Omega}^0-\tilde{\mathcal{V}})+2\tilde{\kappa} \Big )e^{2\tilde{\kappa} t}\\
			&\hspace{1cm} =\Big(\|\Omega^0-\mathcal{V}\|_\infty+2\kappa \Big)e^{-t/m}+ 4m\kappa(\|\mathcal{V}\|_\infty+2\kappa) +4m\kappa \Big (\|\Omega^0-\mathcal{V}\|_\infty+2\kappa\Big )e^{4\kappa t}
	\end{align*}
	so that by \eqref{eq:equal-norms},
	\[
	\|\dot\Theta(m,t)-\dot\Theta(0,t)\|_\infty<\frac 12\Big(\|\Omega^0-\mathcal{V}\|_\infty+2\kappa \Big)e^{-t/m}+ 2m\kappa(\|\mathcal{V}\|_\infty+2\kappa) +2m\kappa \Big (\|\Omega^0-\mathcal{V}\|_\infty+2\kappa\Big )e^{4\kappa t}.
	\]
	\item (Convergence of $\Phi^{(n)}(m, \cdot)$ with $n \geq 2$):
	For $n\ge 2$, we have
	\begin{align*}
		&\|\Phi^{(n)}(m,t)-\Phi^{(n)}(0,t)\|_\infty\\
		& \hspace{0.5cm} \le (n-1)!\left(2\tilde{\kappa}+\|\tilde{\Omega}^0\|_\infty+\|\tilde{\mathcal{V}}\|_\infty+\frac{9}{8m}\right)^{n}\left(1+\frac tm\right)^{n}e^{-t/m}\\
		&  \hspace{0.7cm} + \frac 98m\tilde{\kappa} \cdot  n!e^{2\tilde{\kappa}t}\left(2\tilde{\kappa}+\mathcal{D}(\tilde{\Omega}^0)+\mathcal{D}(\tilde{\mathcal{V}})+\frac{9}{8m}\right)^{n}\left(1+\frac tm\right)^{n}e^{-t/m}\\
		& \hspace{0.7cm} + \frac 34m\tilde{\kappa} \cdot  (n+1)!e^{2\tilde{\kappa} t} \Big(2\tilde{\kappa}+\mathcal{D}(\tilde{\Omega}^0)+\mathcal{D}(\tilde{\mathcal{V}}) \Big)^{n}(1-e^{-t/m})\\
		& \hspace{0.5cm} \le (n-1)!\left(4\kappa+\|\Omega^0\|_\infty+\|\mathcal{V}\|_\infty+\frac{9}{8m}\right)^{n}\left(1+\frac tm\right)^{n}e^{-t/m}\\
		&  \hspace{0.7cm} + 9m\kappa \cdot  n!2^{n-2}e^{4\kappa t}\left(2\kappa+\|\Omega^0\|_\infty+\|\mathcal{V}\|_\infty+\frac{9}{16m}\right)^{n}\left(1+\frac tm\right)^{n}e^{-t/m}\\
		& \hspace{0.7cm} + 3m \kappa \cdot  (n+1)!2^{n-1}e^{4\kappa t} \Big(2\kappa+\|\Omega^0\|_\infty+\|\mathcal{V}\|_\infty \Big)^{n}(1-e^{-t/m}),
	\end{align*}
	so that
		\begin{align*}
		&\|\Theta^{(n)}(m,t)-\Theta^{(n)}(0,t)\|_\infty\\
		&\hspace{0.5cm}\stackrel{\mathclap{\eqref{eq:equal-norms}}}{\le} \frac 12\|\Phi^{(n)}(m,t)-\Phi^{(n)}(0,t)\|_\infty\\
		& \hspace{0.5cm} \le \frac 12(n-1)!m^{-n}\left(4m\kappa+m\|\Omega^0\|_\infty+m\|\mathcal{V}\|_\infty+\frac{9}{8}\right)^{n}\left(1+\frac tm\right)^{n}e^{-t/m}\\
&  \hspace{0.7cm} + 9m\kappa \cdot  n!2^{n-3}e^{4\kappa t}m^{-n}\left(2m\kappa+m\|\Omega^0\|_\infty+m\|\mathcal{V}\|_\infty+\frac{9}{16}\right)^{n}\left(1+\frac tm\right)^{n}e^{-t/m}\\
& \hspace{0.7cm} + 3m \kappa \cdot  (n+1)!2^{n-2}e^{4\kappa t} \kappa^n \Big(2+\frac{\|\Omega^0\|_\infty}{\kappa}+\frac{\|\mathcal{V}\|_\infty}\kappa \Big)^{n}(1-e^{-t/m}).
	\end{align*}
	One can check that this is in the stated form.
\end{enumerate}
\end{proof}

The proper exponential growth rate is likely given by $2\kappa$ as in the case for the Kuramoto model in \cite{cho2025inertia}, and our proof via the embedding likely introduces an unnecessary factor of 2 in the exponent. However,  this does not
affect the qualitative form of the bounds, so we do not pursue this optimization here.

\begin{remark}
The same embedding does \emph{not} reduce
Theorems~\ref{thm:pathwise} or~\ref{thm:qualitative_mto0} to the
corresponding Kuramoto results of~\cite{cho2025inertia,cho2025quantitative}. Theorem \ref{thm:pathwise} is not a simple reduction of \cite[Theorem 1.1]{cho2025quantitative} because $R_0^{3/2}$ cannot be bounded by $R_0^{2}$, and Theorem \ref{thm:qualitative_mto0} is not a simple reduction of \cite[Theorem 1.2]{cho2025inertia} because the limiting order parameter there is not strong enough (this is because, in Section \ref{sec:qualitative}, the first-order Winfree
result \cite{ha2015emergence}, convergence of the
first-order trajectory to a death state, has no Kuramoto analogue).
Consequently, while Proposition~\ref{prop:mto0quant} is essentially a
corollary of~\cite[Proposition~3.2]{cho2025inertia},
Theorems~\ref{thm:pathwise} and~\ref{thm:qualitative_mto0} require
genuinely Winfree-specific arguments, which are carried out in
Sections~\ref{sec:pathwise-proof} and~\ref{sec:qualitative}.
\end{remark}

\section{Partial oscillator death and proof of Theorem~\ref{thm:pathwise}}
\label{sec:pathwise-proof}

Throughout this section $\Theta(t)$ denotes the solution
to~\eqref{Winfree}, and we write $R(t)=R(\Theta(t))$. The proof of
Theorem~\ref{thm:pathwise} is organized into three stages, namely initial
layer, condensation, and persistence, in direct parallel with the
treatment of the inertial Kuramoto model
in~\cite{cho2025quantitative}. The technical backbone consists of a \emph{partial oscillator death} result (Proposition \ref{prop:pod}), which we prove via a \emph{partial trapping lemma} (Lemma~\ref{lem:single}).

\subsection{A priori partial trapping}

We first begin by proving that under an a priori lower bound on the order parameter, one can control a partial cluster. Below, $\mathcal{B}\subset [N]$ denotes a subset of oscillators which we wish to control.

\begin{lemma}[A priori partial trapping]\label{lem:single}
Let $\mathcal{B}\subset [N]$, $\eta>0$, $\rho>0$, and $T\in (\eta m,\infty]$, and suppose $R_\mathcal{B}(t)\ge\rho$ for
$t\in[\eta m,T)$. Let $x\in (-1,1)$ satisfy
\begin{equation}\label{eq:xbound}
1-x^2\ge\left(\frac{\|\Omega_\mathcal{B}^0\|_\infty e^{-\eta}+\|\mathcal V_\mathcal{B}\|_\infty+2\kappa\left(\frac{|\mathcal{B}|}{N}m\left(\|\Omega_\mathcal{B}^0\|_\infty(\eta\vee 1)e^{-(\eta\vee 1)}+ M_{F,\mathcal{B}}\right)+\frac{N-|\mathcal{B}|}{N}\right)}{\kappa\rho(1-e^{-\eta})}\right)^{\!2},
\end{equation}
where $z\vee w\coloneqq \max\{z,w\}$.
Let $y\in[-|x|,|x|]$.
\begin{enumerate}
\item If $i\in \mathcal{B}$ is such that $\cos\theta_i(t_0)\ge y$ for
some $t_0\in[\eta m,T)$, then $\cos\theta_i(t)\ge y$ for all
$t\in[t_0,T)$. In particular, $\mathrm{osc}_{[\eta m,T)} \theta_i<2\pi$, where $\mathrm{osc}_{I}f(t)\coloneqq \sup_{t\in I}f(t)-\inf_{t\in I} f(t)$ for an interval $I$ and a function $f:I\to \R$.
\item For all $i\in \mathcal{B}$, we have $\mathrm{osc}_{[\eta m,T)} \theta_i<4\pi$.
\end{enumerate}
\end{lemma}
\begin{remark}\,
\begin{enumerate}
\item Later, we will make the choice $\eta=1$, in which case $\eta\vee 1=1$.
\item What is needed in the proof of Proposition \ref{prop:pod}, and eventually Theorem \ref{thm:pathwise}, is the case $\mathcal{B}=[N]$. However, as partial death criteria are of independent interest, we state Lemma \ref{lem:single} in the above generality.
\end{enumerate}

\end{remark}
\begin{proof}[Proof of Lemma \ref{lem:single}]
	\,
\begin{enumerate}
\item 
If $\cos\theta_i(t)>y$ for all $t\in [t_0,T)$, then there is nothing to prove. Otherwise, let $t_1\in[t_0,T)$ be the infimal time  such that
$\cos\theta_i(t_1)\le y$; this necessitates $\cos\theta_i(t_1)=y$ by continuity. We claim that $\frac{d}{dt}\cos\theta_i(t_1)>0$. Indeed, at $t=t_1$, $|\sin\theta_i(t_1)|=\sqrt{1-y^2}\ge\sqrt{1-x^2}$. Using
Lemma~\ref{lem:orderlowering} at $t_1$,
\begin{align*}
\bigl.\tfrac{d}{dt}\cos\theta_i\bigr|_{t=t_1}
&=-\sin\theta_i(t_1)\dot\theta_i(t_1)\\
&\ge \kappa R_\mathcal{B}(t_1)\sin^2\theta_i(t_1)(1-e^{-t_1/m})\\
&\quad-|\sin\theta_i(t_1)|\bigl(\|\Omega_\mathcal{B}^0\|_\infty e^{-t_1/m}+\|\mathcal V_\mathcal{B}\|_\infty(1-e^{-t_1/m})\bigr)\\
&\quad-|\sin\theta_i(t_1)|2\kappa (1-e^{-t_1/m})\left(\frac{|\mathcal{B}|}{N}\bigl(\|\Omega_\mathcal{B}^0\|_\infty t_1 e^{-t_1/m}+m M_{F,\mathcal{B}}(1-e^{-t_1/m})^2\bigr)+\frac{N-|\mathcal{B}|}{N}\right).
\end{align*}

Dividing by $|\sin\theta_i(t_1)|=\sqrt{1-y^2}(\ge\sqrt{1-x^2}>0)$ and using
$R_\mathcal{B}(t_1)\ge\rho$, $t_1\ge\eta m$,
\begin{align*}
\frac{1}{\sqrt{1-y^2}}\bigl.\tfrac{d}{dt}\cos\theta_i\bigr|_{t=t_1}
&\ge \kappa\rho\sqrt{1-y^2}(1-e^{-\eta})-\|\Omega_\mathcal{B}^0\|_\infty e^{-\eta}-\|\mathcal V_\mathcal{B}\|_\infty\\
&\quad-2\kappa\left(\frac{|\mathcal{B}|}{N}m\left(\|\Omega_\mathcal{B}^0\|_\infty(\eta\vee 1)e^{-(\eta\vee 1)}+ M_{F,\mathcal{B}}\right)+\frac{N-|\mathcal{B}|}{N}\right)\\
&>0,
\end{align*}
where the first inequality uses the fact applied to $s=t_1/m$ that the function $s\mapsto se^{-s}$ is unimodal with maximum at $s=1$, so that $s\ge \eta$ implies that $se^{-s}\le (\eta\vee 1)e^{-(\eta\vee 1)}$, and
the second inequality is \eqref{eq:xbound} combined with $1-y^2\ge 1-x^2$. Thus
$\frac{d}{dt}\cos\theta_i(t_1)>0$.

Since $\cos\theta_i(t_0)\ge y$ yet $t_1$ is the first time such that $\cos\theta_i(t_1)\le y$, the inequality $\frac{d}{dt}\cos\theta_i(t_1)>0$ implies that $t_1=t_0$ (since otherwise $t_1>t_0$ and there are times slightly smaller than $t_1$ at which $\cos \theta_i<y$, violating the minimality of $t_1$) So, it must be that $\cos\theta_i(t_0)=y$ and $\cos\theta_i(t)>y$ for $t\in [t_0,t_0+\varepsilon)\subset [t_0,T)$ for some small $\varepsilon$.

If there were another time $t_2\in [t_0+\varepsilon,T)$ for which $\cos\theta_i(t_2)\le y$, then by making $t_2$ the earliest such time, we must have $\cos\theta_i(t_2)=y$, and the same computation as above shows that $\frac{d}{dt}\cos\theta_i(t_2)>0$, giving a contradiction (as there would be times slightly smaller than $t_2$ at which $\cos\theta_i<y$, violating the minimality of $t_2$).

By this exit-time argument, we have $\cos\theta_i(t)\ge y$ for all $t\in[t_0,T)$ (and in fact, the stronger statement that $\cos\theta_i(t)>y$ for $t\in (t_0,T)$).

The latter statement follows from $y>-1$, and the continuity of $\theta_i$ in $t$. More explicitly, $\theta_i|_{[t_1,T)}$ is confined to a single component of $\{\cos\theta\ge -|x|\}$, a closed
interval of length $2(\pi-\arccos|x|)<2\pi$; hence
$\mathrm{osc}_{[t_1,T)}\,\theta_i<2\pi$. 
\item Fix $i\in\mathcal{B}$. Define
$t^{*}:=\inf\{t\in[\eta m,T):\cos\theta_i(t)\ge -|x|\}$
(set $t^{*}=T$ if the set is empty).

If $t^{*}=T$, then $\cos\theta_i(t)<-|x|$ for every $t\in[\eta m,T)$;
since $|x|<1$, the continuous lift $\theta_i$ is confined to a single
connected component of $\{\theta\in\R:\cos\theta<-|x|\}$, an open
interval of length $2\arccos|x|<2\pi$.

If $t^{*}<T$, then by continuity $\cos\theta_i(t^{*})\ge -|x|$, and
part~(1) with $y=-|x|$ gives $\cos\theta_i(t)\ge -|x|$ for all
$t\in[t^{*},T)$. The lift $\theta_i|_{[t^{*},T)}$ is therefore
confined to a single component of $\{\cos\theta\ge -|x|\}$, a closed
interval of length $2(\pi-\arccos|x|)<2\pi$; hence
$\mathrm{osc}_{[t^{*},T)}\,\theta_i<2\pi$. On $[\eta m,t^{*})$ we
have $\cos\theta_i<-|x|$, so by the same connected-component argument
$\mathrm{osc}_{[\eta m,t^{*}]}\,\theta_i\le 2\arccos|x|<2\pi$.
Sub-additivity of oscillation gives
\[
\mathrm{osc}_{[t_1,T)}\,\theta_i\le \mathrm{osc}_{[\eta m,t^{*}]}\,\theta_i+\mathrm{osc}_{[t^{*},T)}\,\theta_i
<2\pi+2\pi=4\pi.\qedhere
\]

\end{enumerate}
\end{proof}

\subsection{Criterion for partial oscillator death}

We now state and prove the main criterion for partial oscillator death. This involves an order parameter bootstrapping argument of \cite{ryoo2026oscillator}. Below, now there are two subsets $\mathcal{A}\subset \mathcal{B}\subset [N]$. The smaller set $\mathcal{A}$ is a subset that we know that the oscillators are in some neighborhood of $0$ and which we wish to have detailed control over, and the larger set $\mathcal{B}$ are oscillators whose locations we do not know precisely but which we wish to have some loose control over.

\begin{proposition}[Criterion for partial oscillator death]\label{prop:pod}
Let $\Theta(t)$ be the solution to~\eqref{Winfree}, and let
$\eta>0$, $t_0\ge\eta m$, $\rho\in(0,2]$, $x\in (-1,1)$, and
$\mathcal A\subset\mathcal{B}\subset [N]$. Assume \eqref{eq:xbound} and assume that at time $t_0$,
\begin{equation}\label{eq:pod-init}
\cos\theta_i(t_0)\ge x\ \ \text{for every }i\in\mathcal A,\qquad
\frac1N\sum_{i\in\mathcal A}\bigl(1+\cos\theta_i(t_0)\bigr)>\frac{2\rho}{1+|x|}.
\end{equation}
Then:
\begin{enumerate}
\item $R_\mathcal{A}(t)\ge\rho$ for every $t\ge t_0$;
\item $\cos\theta_i(t)\ge x$ for every $i\in\mathcal A$ and $t\ge t_0$, so that $\operatorname{osc}_{ [t_0,\infty)}\theta_i<2\pi$ for every
$i\in\mathcal A$;
\item $\operatorname{osc}_{[t_0,\infty)}\theta_i<4\pi$ for every
$i\in \mathcal{B}$.
\end{enumerate}
\end{proposition}
\begin{proof}
	\,
	
\emph{Items (1) and (2).} Define
$\mathcal T:=\{T>t_0\colon R_\mathcal{A}(t)\ge\rho\text{ for all }t\in[t_0,T)\}$.
By~\eqref{eq:pod-init},
$R_\mathcal{A}(t_0)\ge\frac1N\sum_{i\in\mathcal A}(1+\cos\theta_i(t_0))>\frac{2\rho}{1+|x|}>\rho$,
so by continuity $\mathcal T\ne\emptyset$; set $T^\star:=\sup\mathcal T\in(t_0,\infty]$.
For $t\in[t_0,T^\star)$ we have $R_\mathcal{B}(t)\ge R_\mathcal{A}(t)\ge\rho$, so
Lemma~\ref{lem:single} (1) applies; it yields
$\cos\theta_i(t)\ge x$ for $i\in\mathcal A$, $t\in[t_0,T^\star)$, and $\operatorname{osc}_{[t_0,T^\star)}\theta_i<2\pi$, which
is item~(2) but restricted to $[t_0,T^\star)$. So to prove items (1) and (2), it is enough to prove $T^\star=\infty$.

We observe that for $i\in \mathcal{A}$ and $t\in [t_0,T^\star)$, we have
$1+\cos\theta_i(t)\ge\frac{1+|x|}{2}(1+\cos\theta_i(t_0))$.
To see this, consider two cases.
If $\cos\theta_i(t_0)\le|x|$, then
Lemma~\ref{lem:single} applied with $y:=\cos\theta_i(t_0)\in[-|x|,|x|]$
gives $\cos\theta_i(t)\ge\cos\theta_i(t_0)$, so
$1+\cos\theta_i(t)\ge 1+\cos\theta_i(t_0)\ge\frac{1+|x|}{2}(1+\cos\theta_i(t_0))$
since $\frac{1+|x|}{2}\le 1$.
If $\cos\theta_i(t_0)>|x|$, then Lemma~\ref{lem:single} with $y:=|x|$
gives $\cos\theta_i(t)\ge|x|$, so
$1+\cos\theta_i(t)\ge 1+|x|\ge\frac{1+|x|}{2}(1+\cos\theta_i(t_0))$
since $1+\cos\theta_i(t_0)\le 2$.

Summing, we have
\begin{align*}
R_\mathcal{A}(t)\ge\frac1N\sum_{i\in\mathcal A}(1+\cos\theta_i(t))\ge\frac{1+|x|}{2}\cdot\frac1N\sum_{i\in\mathcal{A}}(1+\cos\theta_i(t_0))>\frac{1+|x|}{2}\cdot\frac{2\rho}{1+|x|}=\rho,\quad t\in [t_0,T^\star).
\end{align*}
If $T^\star<\infty$, then by definition of $\mathcal{T}$ and continuity of $R_\mathcal{A}$, we have $R_\mathcal{A}(T^\star)=\rho$. However, continuity of $R_\mathcal{A}$ again gives
\[
R_\mathcal{A}(T^\star)=\lim_{t\to T^\star -}R_\mathcal{A}(t)\ge \frac{1+|x|}{2}\cdot\frac1N\sum_{i\in\mathcal A}(1+\cos\theta_i(t_0))>\frac{1+|x|}{2}\cdot\frac{2\rho}{1+|x|}=\rho,
\]
a contradiction. Hence
$T^\star=\infty$, and items~(1) and~(2) hold for all $t\ge t_0$.

\smallskip
\emph{Item (3).} This follows from Lemma \ref{lem:single}(2).

\end{proof}

\begin{remark}
The threshold condition \eqref{eq:pod-init}, specifically, the
requirement that the $\mathcal A$-cluster sum exceeds
$\tfrac{2\rho}{1+|x|}$, is the analogue for Winfree of the majority
cluster condition of~\cite[Theorem~4.1]{cho2025quantitative}.
The factor $\tfrac{2}{1+|x|}$ comes from the fact that, along the
trajectory, $(1+\cos\theta_i)$ can drop by at most a factor of
$\tfrac{1+|x|}2$ before the trapping mechanism of
Lemma~\ref{lem:single} takes over.
\end{remark}

\subsection{Proof of Theorem~\ref{thm:pathwise}}\label{sec:pathwise-assembly}

Let $\eta:=1$ throughout.

\paragraph{Stage A (initial layer).}

By setting $\delta=\frac 12$ in  Lemma \ref{lem:initlayer}, if
\begin{equation}\label{eq:quant-check-1}
m\|\Omega^0\|_\infty(1-e^{-1})+M_F\cdot me^{-1}<\frac 12 R_0,
\end{equation}
then
\begin{equation}\label{eq:initial-layer}
R(t)>\frac 12 R_0,\qquad t\in [0,m].
\end{equation}

Because
\begin{equation}\label{eq:M_F-bound}
	M_F\le \|\mathcal{V}\|_\infty +2\kappa <(2+aR_0^{3/2})\kappa,
\end{equation}
we have, using $R_0\le 2$,
\begin{align*}
m\|\Omega^0\|_\infty(1-e^{-1})+M_F\cdot me^{-1}&<m\kappa \cdot \frac{\|\Omega^0\|_\infty}{\kappa}\cdot (1-e^{-1}) + (2+aR_0^{3/2})\kappa\cdot me^{-1}\\
&<bc R_0^3+(2+aR_0^{3/2})e^{-1}b R_0^{3/2}\\
&\le \left(4bc+(2+2\sqrt{2}a)e^{-1}\sqrt{2}b\right) R_0
\end{align*}
so that \eqref{eq:quant-check-1} is satisfied if
\begin{equation}\label{eq:quant-check-1-abc}
	4bc+(2+2\sqrt{2}a)e^{-1}\sqrt{2}b<\frac 12.
\end{equation}
With our choice of $a=1/50$, $b=1/80$, and $c=1/20$, this computes to
\[
0.0159\cdots<\frac 12.
\]

\paragraph{Stage B (condensation at $t_0=m$).}
Let $\mu\in(0,\frac{R_0}{2}]$ be a free parameter,
to be chosen later.
 Define
\begin{equation}\label{def:cA}
\mathcal A\coloneqq\{i\in[N]\colon\cos\theta_i(m)\ge-1+\mu\}.
\end{equation}
 We claim that 
\begin{equation}\label{eq:cA-lbound}
	\frac 1N\sum_{i\in \mathcal{A}}(1+\cos\theta_i(m))\ge\frac{2(R(m)-\mu)}{2-\mu }>\frac{R_0-2\mu}{2-\mu}.
\end{equation}
To prove the claim, we first observe by definition \eqref{order_parameter} that
 \begin{equation}\label{sum}
 	\frac 1N\sum_{i\in \mathcal{A}}(1+\cos\theta_i(m))+\frac 1N\sum_{i\in \{1,\cdots,N\}\setminus\mathcal{A}}(1+\cos\theta_i(m))=R(m).
 \end{equation}
 Then, using
 \begin{equation}\label{sizebound}
 	\begin{cases}
 		1+\cos\theta_i(m)\le 2, & \mbox{for } i\in \mathcal{A},\\
 		1+\cos\theta_i(m)\le \mu ,&\mbox{for } i\in \{1,\cdots,N\}\setminus\mathcal{A},
 	\end{cases}
 \end{equation}
 we derive
 \begin{align*}
 	\mu \cdot \frac 1N\sum_{i\in \mathcal{A}}(1+\cos\theta_i(m))\stackrel{\eqref{sizebound}}{\le}& \mu \cdot \frac{2|\mathcal{A}|}{N}=2\mu -2\mu \cdot \frac{|\{1,\cdots,N\}\setminus\mathcal{A}|}{N}\\
 	\stackrel{\eqref{sizebound}}{\le} &2\mu -\frac 2N\sum_{i\in \{1,\cdots,N\}\setminus \mathcal{A}}(1+\cos\theta_i(m))\\
 	\stackrel{\eqref{sum}}{=}&2\mu -2R(m)+\frac 2N\sum_{i\in  \mathcal{A}}(1+\cos\theta_i(m)),
 \end{align*}
 from which the above claim \eqref{eq:cA-lbound} follows.

We apply Proposition~\ref{prop:pod} at $t_0=\eta m=m$, with
$\mathcal A$ as in~\eqref{def:cA} and $\mathcal{B}=[N]$. Choose $\mu=\rho=\frac{R_0}{4}$ and $x=-1+\mu=-1+\frac{R_0}{4}\in (-1,-\frac 12]$, so that $|x|=1-\frac{R_0}{4}$. By definition of $\mathcal{A}$, the first statement of \eqref{eq:pod-init} follows. The second statement of \eqref{eq:pod-init} follows by \eqref{eq:cA-lbound} and our choice of $\mu,\rho,x$.

Condition \eqref{eq:xbound} with $\mathcal{B}=[N]$ and our choice of parameters reads
\begin{equation}\label{eq:balance}
	\frac{R_0}{4}\left(2-\frac{R_0}{4}\right)\ge\left(\frac{\|\Omega^0\|_\infty e^{-1}+\|\mathcal V\|_\infty+2m\kappa\|\Omega^0\|_\infty e^{-1}+2m\kappa M_F}{\kappa R_0(1-e^{-1})/4}\right)^{\!2}
\end{equation}
which is equivalent to
\[
	\frac{R_0^3}{64}\left(2-\frac{R_0}{4}\right)(1-e^{-1})^2\ge\left(\frac{\|\Omega^0\|_\infty}\kappa e^{-1}+\frac{\|\mathcal V\|_\infty}\kappa+2m\kappa\cdot \frac{\|\Omega^0\|_\infty}{\kappa} e^{-1}+2m\kappa \frac{M_F}{\kappa}\right)^{\!2}.
\]
But by $R_0\le 2$,
\[
\frac{R_0^3}{64}\left(2-\frac{R_0}{4}\right)(1-e^{-1})^2\ge \frac{3R_0^3}{128}(1-e^{-1})^2,
\]
while by \eqref{eq:M_F-bound},
\begin{align*}
&\left(\frac{\|\Omega^0\|_\infty}\kappa e^{-1}+\frac{\|\mathcal V\|_\infty}\kappa+2m\kappa\cdot \frac{\|\Omega^0\|_\infty}{\kappa} e^{-1}+2m\kappa \frac{M_F}{\kappa}\right)^{\!2}\\
&\le \left(c R_0^{3/2}e^{-1}+aR_0^{3/2}+2bR_0^{3/2}\cdot cR_0^{3/2}e^{-1}+2bR_0^{3/2}\cdot (2+aR_0^{3/2})\right)^2\\
&\le \left(e^{-1}c+a+4\sqrt{2}e^{-1}bc+4b (1+a\sqrt{2})\right)^2 R_0^3
\end{align*}
so for \eqref{eq:balance} to hold, it is enough that
\begin{equation}\label{eq:quant-check-2-abc}
e^{-1}c+a+4\sqrt{2}e^{-1}bc+4b (1+a\sqrt{2})\le \frac{\sqrt{3}}{8\sqrt{2}}(1-e^{-1}).
\end{equation}
With our choice of $a=1/50$, $b=1/80$, and $c=1/20$, this computes to
\[
0.0911<\frac{\sqrt{3}}{8\sqrt{2}}(1-e^{-1})\approx 0.0968.
\]
Note that there's a tighter margin than \eqref{eq:quant-check-1-abc}; this condition \eqref{eq:quant-check-2-abc} tends to be the bottleneck when optimizing over $a,b,c$.
\paragraph{Stage C (persistence).}

All hypotheses of Proposition~\ref{prop:pod} are
met at $t_0=m$ with $\mathcal A$ from~\eqref{def:cA} and $\mathcal{B}=[N]$,
$\mu=R_0/4$, $x=-1+\mu$, $\rho=R_0/4$. Its conclusion gives:
\begin{enumerate}
\item $R(t)\ge R_0/4$ for all $t\ge m$. Combined with the initial-layer
bound $R(t)\ge R_0/2$ for $t\in[0,m]$ (which follows
from~\eqref{eq:initial-layer}), we obtain $\inf_{t\ge 0}R(t)\ge R_0/4$,
proving part~(2) of Theorem~\ref{thm:pathwise}.
\item For every $i\in[N]$, we have $\sup_{t\ge t_0}\theta_i(t)-\inf_{t\ge t_0}\theta_i(t)<4\pi$. So, the trajectory is uniformly bounded, and
Proposition~\ref{prop:lojasiewicz} applies and yields
$\dot\theta_i(t)\to 0$ and
$\theta_i(t)\to\theta_i^\infty$, which is part~(1) of
Theorem~\ref{thm:pathwise}. \qed
\end{enumerate}

\begin{remark}[Sharpness in $\mu$ and $\delta$]\label{rmk:R032}
The choice $\mu=R_0/4$ in Stage~C is a convenient but non-optimal
midpoint. Optimizing $\mu\in(0,R_0)$ and $\delta\in (0,1)$ can give different constants in the statement of Theorem \ref{thm:pathwise}. For example, one may improve the
constant on the right-hand side of Theorem \ref{thm:pathwise}(2) from $R_0/4$ to
$R_0(1-\varepsilon)$ for any $\varepsilon>0$ by taking $a,b,c$ very small: one simply takes $\delta\to 1$ and $\mu\to 0$. One might be able to take $a$ arbitrarily close to $\frac 12$ (in the range $R_0\in (0,1]$) at the expense of smaller $b,c,$ by choosing the $\mu$ as in the proof of \cite[Corollary 16]{ryoo2026oscillator}:
\[
\mu = \frac{3+R_0-\sqrt{R_0^2-2R_0+9}}{4}.
\]

However, we stress that the scaling exponent $3/2$ on
$R_0$ in the hypothesis is \emph{not} sensitive to this choice: it
comes from the $\sqrt{\mu(2-\mu)}\sim\sqrt{R_0}$ factor on the
left-hand side of~\eqref{eq:balance} times the $\rho\sim R_0$ factor on
the right-hand side. This is specific to the Winfree model and it has
no analogue in Kuramoto, where the corresponding computation yields
$R_0^{2}$ (see~\cite[Theorem~1.1]{cho2025quantitative}).
\end{remark}

\section{Proof of Theorem~\ref{thm:qualitative_mto0}}\label{sec:qualitative}

We now turn to the qualitative zero-inertia synchronization theorem, Theorem \ref{thm:qualitative_mto0}. Let $\Theta^0\in\R^N$ satisfy $\theta_i^0\in (-\pi,\pi)$ for all $i\in [N]$. Let $\varepsilon\in (0,1)$.

We are to show that there exist positive numbers $a,b,c$, depending only on $\|\Theta^0\|\in [0,\pi)$ and
$\varepsilon$, such that if the initial velocities $\Omega^0$ and
system parameters satisfy
\begin{equation}\label{eq:qual-abc}
	\frac{\|\mathcal{V}\|_\infty}{\kappa}<a,\qquad
m\kappa<b,\qquad 	\frac{\|\Omega^0\|_\infty}{\kappa}<c,
\end{equation}
then oscillator death occurs for the solution $\Theta(t)$ of
\eqref{Winfree} and
\[
\lim_{t\to\infty}R(t)>2-\varepsilon.
\]

Recall that we write $\Theta(m,t)$ for the solution of~\eqref{Winfree} with initial position and velocity data $\Theta^0,\Omega^0$, and $\Theta(0,t)$ for the solution of~\eqref{GenWinfree}--\eqref{standard} with initial
position data $\Theta^0$.

\subsection{Step 1. The first-order Winfree theory}
We recall the following result on the first-order solution $\Theta(0,t)$.
\begin{theorem}[{\cite[special case of Theorem 2.2]{ha2015emergence}}]\label{thm:HPR}
	For $\alpha\in \left(\frac{\pi}{3},\pi\right)$, let $\alpha^\infty\in \left(0,\frac{\pi}{3}\right)$ be the unique solution to $\sin\alpha^\infty(1+\cos\alpha^\infty)=\sin\alpha(1+\cos\alpha)$ (note that the function $s\mapsto \sin s(1+\cos s)$ is unimodal, increasing on $[0,\pi/3]$ and decreasing on $[\pi/3,\pi]$).
	
	\begin{enumerate}
		\item (Existence and uniqueness of equilibrium) Then system \eqref{GenWinfree}--\eqref{standard} with parameters $\mathcal{V}$ and $\kappa$ satisfying
		\[
		\frac{\|\mathcal{V}\|_\infty}{\kappa}<\sin\alpha(1+\cos\alpha)
		\]
		has a unique equilibrium $\Theta^\infty$ in $[-\alpha,\alpha]^N$.
		Furthermore, $\Theta^\infty\in (-\alpha^\infty,\alpha^\infty)^N$.
		\setcounter{name}{\value{enumi}}
	\end{enumerate}
	Moreover, let $\Theta(0,t)=\{\theta_i(0,t)\}_{i=1}^N$ be the solution to \eqref{GenWinfree}--\eqref{standard} with initial data $\Theta^0$ such that $\theta_i^0\in [-\alpha,\alpha]$, $i=1,\cdots,N$. Then $\Theta(0,t)\to \Theta^\infty$ exponentially as $t\to\infty$. More precisely,
	\begin{enumerate}
		\setcounter{enumi}{\value{name}}
		\item (Finite-time entrance into stable region) there exists a time $T\le \frac{\pi}{\kappa\sin\alpha(1+\cos\alpha)-\|\mathcal{V}\|_\infty}$ such that $\Theta(0,t)\in (-\alpha^\infty,\alpha^\infty)^N$ for $t\ge T$, and we can take $T=0$ if $\Theta^0\in (-\alpha^\infty,\alpha^\infty)^N$; also,
		\item (Exponential convergence to equilibrium) we have that
		\[
		\|\Theta(0,t)-\Theta^\infty\|_{\ell_1^N}\le \|\Theta(0,T)-\Theta^\infty\|_{\ell_1^N}\exp\left[-\kappa (2\cos\alpha^\infty-1)(\cos\alpha^\infty+1)(t-T)\right],\quad t\ge T.
		\]
	\end{enumerate}
\end{theorem}
We will only use statements (1) and (2) but not (3).

Since $\theta_i^0\in (-\pi,\pi)$ for all $i\in [N]$, we may choose an $\alpha\in \left(\frac{\pi}{3},\pi\right)$, depending only on $\|\Theta^0\|$ and $\varepsilon>0$, such that
\[
\theta_i^0\in (-\alpha,\alpha)~\forall i\in [N],\qquad \mathrm{and}\qquad \cos (2\alpha^\infty)>1-\frac{\varepsilon}2
\]
(the latter being possible since $\alpha^\infty\to 0$ as $\alpha\to\pi$), 
and let $a>0$ being small enough (depending on $\alpha$) so that
\[
a\le\frac 12 \sin \alpha (1+\cos\alpha).
\]
Then $\frac{\|\mathcal{V}\|_\infty}{\kappa}<a<\sin\alpha(1+\cos\alpha)$, so that by Theorem \ref{thm:HPR}, there exists a time
\[
\frac{\pi}{\kappa\sin\alpha(1+\cos\alpha)-\|\mathcal{V}\|_\infty}\le \frac{\pi}{\kappa\sin\alpha(1+\cos\alpha)-a\kappa}\le \frac{2\pi}{\kappa\sin\alpha(1+\cos\alpha)}=T
\]
such that
\[
\Theta(0,T)\in (-\alpha^\infty,\alpha^\infty)^N.
\]
By taking $b$ sufficiently small in \eqref{eq:qual-abc}, we may ensure that $T\ge m$.

\subsection{Step 2. Tikhonov approximation}

By \eqref{eq:tik-phase},
\begin{equation}
	\|\Theta(m,T)-\Theta(0,T)\|_\infty < \frac 12 m \Big(\|\Omega^0-\mathcal{V}\|_\infty+2\kappa \Big )e^{4\kappa T}.
\end{equation}
Therefore, if
\[
\frac 12 m \Big(\|\Omega^0-\mathcal{V}\|_\infty+2\kappa \Big )e^{4\kappa T}<\alpha^\infty,
\]
which is possible for small enough $a,b,c>0$ (depending on $\alpha$) since
\[
\frac 12 m \Big(\|\Omega^0-\mathcal{V}\|_\infty+2\kappa \Big )e^{4\kappa T}\le \frac 12 m\kappa \Big(\frac{\|\Omega^0\|_\infty}{\kappa}+\frac{\|\mathcal{V}\|_\infty}{\kappa}+2 \Big )e^{4\kappa T}<\frac{b}{2}(a+c+2)\exp\left(\frac{2\pi}{\sin\alpha(1+\cos\alpha)}\right),
\]
we have that
\[
\Theta(m,T)\in (-2\alpha^\infty,2\alpha^\infty)^N.
\]

\subsection{Step 3. Oscillator death criterion}

Set $\mathcal{A}=\mathcal{B}=[N]$, $\eta=1$, $t_0=T$, $\rho=2-\varepsilon$, $x=\cos(2\alpha^\infty)$ in Proposition \ref{prop:pod}. Since $\cos\theta_i(T)>\cos(2\alpha^\infty)=x$ for $i\in [N]$, and
\[
\frac{1}{N}\sum_{i\in [N]}(1+\cos\theta_i(T))>1+\cos(2\alpha^\infty)=1+x>\frac{(2-\varepsilon/2)^2}{1+|x|}>\frac{2\rho}{1+|x|}
\]
(the second inequality following from $x=\cos(2\alpha^\infty)>1-\frac\varepsilon 2>0$), equation  \eqref{eq:pod-init} is satisfied. So we have that, as long as \eqref{eq:xbound} holds:
	\begin{equation}\label{eq:verify-0}
	\sin(2\alpha^\infty)\ge \frac{\|\Omega^0\|_\infty e^{-1}+\|\mathcal V\|_\infty+2m\kappa\|\Omega^0\|_\infty e^{-1}+2m\kappa M_F}{\kappa\rho(1-e^{-1})}
\end{equation}
then oscillator death holds and $R(m,t)\ge \rho = 2-\varepsilon$, as desired.

To verify \eqref{eq:verify-0}, we note that its right-hand side is bounded via \eqref{eq:qual-abc} by
\begin{align*}
&\frac{\|\Omega^0\|_\infty e^{-1}+\|\mathcal V\|_\infty+2m\kappa\|\Omega^0\|_\infty e^{-1}+2m\kappa M_F}{\kappa\rho(1-e^{-1})}\\
&\le \frac{(\|\Omega^0\|_\infty/\kappa) e^{-1}+(\|\mathcal V\|_\infty/\kappa)+2b(\|\Omega^0\|_\infty/\kappa) e^{-1}+2b (\|\mathcal{V}\|_\infty/\kappa+2)}{(2-\varepsilon)(1-e^{-1})}\\
&\le \frac{c e^{-1}+a+2bc e^{-1}+2b (a+2)}{(2-\varepsilon)(1-e^{-1})}
\end{align*}
and can be made less than $\sin (2\alpha^\infty)$ by taking $a,b,c>0$ small depending on $\alpha$.\qed

\begin{remark}
Since $\theta_i(t)\in [-2\alpha^\infty,2\alpha^\infty]$ for $t\ge T$, if $\alpha$ is close enough to $\pi$ such that $2\alpha^\infty<\alpha$, then the system \eqref{Winfree} converges to the unique equilibrium given by Theorem \ref{thm:HPR}.
\end{remark}

\section{Comments on volumetric and Lyapunov functional arguments}\label{sec:volume}

In the first-order model \eqref{GenWinfree}-\eqref{standard}, it is known that for Lebesgue-a.e.~initial data $\Theta^0$, $\kappa>2\|\mathcal{V}\|_\infty$ guarantees oscillator death \cite[Theorem 1]{ryoo2026oscillator}. There, the proof was to first establish a bound on the pathwise critical coupling strength \cite[Corollary 16]{ryoo2026oscillator}, as done in Theorem \ref{thm:pathwise}, and then to invoke volumetric arguments regarding the divergence of the vector field defining the first-order Winfree model. We were unable to replicate this result in the inertial Winfree model \eqref{Winfree} and will describe some possible approaches.

The first-order Winfree model \eqref{GenWinfree}--\eqref{standard} can be viewed as a dynamical system on the compact torus $\mathbb{T}^N$. Likewise, the inertial Winfree model~\eqref{Winfree} can be viewed as a dynamical system on the compact phase space given by the
cylinder
\[
\mathcal M\ \coloneqq\ \T^N\times\overline{B_\lambda},\qquad
\overline{B_\lambda}\coloneqq\{p\in\R^N:\|p\|_\infty\le\lambda\},\qquad
\lambda\coloneqq\|\Omega^0\|_\infty+\|\mathcal V\|_\infty+2\kappa,
\]
to which the flow of~\eqref{Winfree} descends in the $\theta$-variables
and which is invariant in the $p$-variables by the speed-limit
Lemma~\ref{lem:speedlim}.

\subsection{A volumetric argument}
\label{subsec:volume}

The first-order model \eqref{GenWinfree}--\eqref{standard} gives the flow of the vector field
\[
X\ = \sum_{i=1}^N\bigl(\nu_i-\kappa R(\theta)\sin\theta_i\bigr)\partial_{\theta_i},
\]
the divergence of which is
\[
\operatorname{div} X = \kappa\left(NR(1-R)+\frac 1N \sum_{i=1}^N \sin^2\theta_i\right) \ge \kappa NR(1-R).
\]
So the flow expands in the region $R<1$. Therefore, excluding a measure-zero set of initial data, the flow must eventually reach a state at which $R\ge 1$, at which we can use the pathwise oscillator death theorem \cite[Corollary 16]{ryoo2026oscillator}. This gives \cite[Theorem 1]{ryoo2026oscillator}.

However, for the second-order model \eqref{Winfree}, on the phase space $\mathbb T^N\times\R^N$ the vector field is given by
\[
X\ =\ \sum_{i=1}^N p_i\,\partial_{\theta_i}
\ +\ \sum_{i=1}^N\frac{1}{m}\bigl(-p_i+\nu_i-\kappa R(\theta)\sin\theta_i\bigr)\partial_{p_i},
\]
where  $p_i\coloneqq\dot\theta_i$, the divergence of which is $-\frac{N}{m}$. Since this is nowhere positive, we can only see that the flow is contractive. Thus we cannot deduce instability of states with small order parameter; it might even by contracting onto a smaller-dimensional submanifold. One may imagine modifying the volume form or Riemannian metric to obtain a better divergence, then restrict to the compact cylinder $\mathcal M\ \coloneqq\ \T^N\times\overline{B_\lambda}$ to prove almost-sure convergence results. 

We provide one obstruction: we show that we cannot work with volume forms depending only on $R$ that give divergences that depend only on $R$, except for the trivial case of Lebesgue measure with constant negative divergence.

\begin{lemma}[Obstruction for $R$-only volume forms]\label{lem:vol-obstr}
	Let $\mu=h(R(\theta))\,d\theta\wedge dp$ with $h\in C^1(\R_{\ge 0};\R_{>0})$.
	Then
	\begin{equation}\label{eq:divR-only}
		\operatorname{div}_\mu X
		\ =\ -\frac{N}{m}\ -\ \frac{h'(R)}{h(R)}\,S(\theta,p),
		\qquad
		S(\theta,p)\coloneqq\frac{1}{N}\sum_{i=1}^N p_i\sin\theta_i.
	\end{equation}
	In particular, $\operatorname{div}_\mu X$ is a function of $R$ alone
	if and only if $h$ is constant, in which case $\operatorname{div}_\mu X\equiv-N/m$.
\end{lemma}
\begin{proof}
	For a volume form $\mu=h\,d\theta\wedge dp$ with $h>0$, one has
	$\operatorname{div}_\mu X=\tfrac{1}{h}\,\nabla\!\cdot\!(hX)$, where
	$\nabla\!\cdot$ is the Euclidean divergence in the coordinates
	$(\theta,p)$. Since $h(R)$ is independent of $p$,
	\[
	\sum_{i=1}^N\partial_{p_i}\!\Bigl[h(R)\cdot\tfrac{1}{m}(-p_i+\nu_i-\kappa R\sin\theta_i)\Bigr]
	\ =\ -\frac{N}{m}h(R).
	\]
	For the $\theta$-derivatives,
	$\partial_{\theta_j}R=-\tfrac{1}{N}\sin\theta_j$, and so
	\[
	\sum_{i=1}^N\partial_{\theta_i}\bigl[h(R)p_i\bigr]
	\ =\ \sum_{i=1}^N p_i\,h'(R)\,\partial_{\theta_i}R
	\ =\ -h'(R)\,S(\theta,p).
	\]
	Dividing by $h(R)$ gives~\eqref{eq:divR-only}. The last assertion is
	immediate: for fixed $\theta$ with some $\sin\theta_{j}\neq 0$, the
	map $p\mapsto S(\theta,p)$ is a nonzero linear function of $p$, so
	\eqref{eq:divR-only} depends only on $R$ iff $h'(R)/h(R)\equiv 0$, i.e.\
	$h$ is constant.
\end{proof}
Motivated by the first-order case, one might ask if $\operatorname{div}_\mu X$ could be \emph{lower bounded} by a function of $R$ alone. However, this might not be possible if the dynamics of \eqref{Winfree} degenerate onto a lower-dimensional submanifold.

\subsection{A conjectural Lyapunov functional}

On the compact, real-analytic manifold with boundary $\mathcal{M}$, if a Lyapunov functional, i.e., a function 
$\mathcal L\colon\mathcal M\to\R$ satisfying
\begin{equation}\label{eq:lyap-strong}
	\frac{d}{dt}\mathcal L(\Theta(t),\dot\Theta(t))\le 0,\qquad
	\frac{d}{dt}\mathcal L=0\iff(\Theta,\dot\Theta)\in\mathcal S,
\end{equation}
existed, then we might be able to obtain a LaSalle-type convergence result via the Haraux--Jendoubi variant of \L ojasiewicz's inequality
(Proposition~\ref{prop:lojasiewicz}), if, say, the function $\mathcal{L}$ satisfied the stronger condition
\begin{equation}\label{eq:lyap-stronger}
-C_1\|\dot\Theta(t)\|^2\le\frac{d}{dt}\mathcal L(\Theta(t),\dot\Theta(t))\le -C_2\|\dot\Theta(t)\|^2
\end{equation}
for some universal constants $C_1,C_2>0$.

\begin{question}[Refined Lyapunov conjecture]\label{conj:lyapunov}
	Denote by $\kappa_c(\mathcal V)$ the critical coupling strength of
	\eqref{Winfree} (as defined
	in~\cite{ryoo2026oscillator}\footnote{This is the coupling strength
		above which phase-locked states exist for the first-order Winfree
		model \eqref{GenWinfree}--\eqref{standard}.}). For every
	$\kappa>\kappa_c(\mathcal V)$, does there exist a function $\mathcal{L}:\mathcal{M}\to \mathbb{R}$ satisfying \eqref{eq:lyap-stronger}?
\end{question}

We remark that the well-known potential  $P(\Theta)=-\sum\nu_k\theta_k-\tfrac{\kappa N}{2}R^2$
of Proposition~\ref{prop:lojasiewicz} is \emph{not} a function on
$\T^N$, because its linear piece has periods $-2\pi\nu_i$ around the
fundamental loops of the torus. We must therefore search for a
Lyapunov functional that is intrinsically torus-periodic, i.e.\ built
from $(\cos\theta,\sin\theta,p)$ alone.

\section{Concluding remarks and open problems}\label{sec:conclusion}

We have proved two synchronization theorems for the second-order
Winfree model: a pathwise oscillator-death theorem
(Theorem~\ref{thm:pathwise}) with explicit $R_0^{3/2}$ scaling, and a
qualitative zero-inertia synchronization theorem
(Theorem~\ref{thm:qualitative_mto0}), both complementary to the
first-order Winfree results of~\cite{ryoo2026oscillator} and the
inertial Kuramoto results
of~\cite{cho2025quantitative,cho2025inertia}. Along the way we proved
a quantitative higher-order Tikhonov theorem
(Proposition~\ref{prop:mto0quant}). We close with a short list of open
problems.

\begin{enumerate}
\item \textbf{$R_0$-independent pathwise constants.}
The constants
$a=\frac{1}{50}$, $b=\frac{1}{80}$, $c=\frac{1}{20}$ of Theorem~\ref{thm:pathwise} are certainly
not sharp. More importantly, the $R_0^{3/2}$-dependence is limiting, as Theorem \ref{thm:pathwise} does not give a uniform bound that works for all initial data. We pose the following conjecture.
\begin{conjecture}\label{conj:weak}
	There exist absolute constants $a,b,c>0$ with the following property. For any initial data
	$(\{\theta_i^0\}_{i=1}^N,\{\omega_i^0\}_{i=1}^N)$ and system parameters $(\{\nu_i\}_{i=1}^N,\kappa,m)$
	satisfying 
	\begin{equation}
		\frac{\|\mathcal{V}\|_\infty}{\kappa}<a,\qquad
		m\kappa<b,\qquad
		\frac{\|\Omega^0\|_\infty}{\kappa}<c,
	\end{equation}
	the solution $\Theta$ to \eqref{Winfree} exhibits oscillator death, i.e., for every $i\in[N]$, the limits
	$\theta_i^\infty:=\lim_{t\to\infty}\theta_i(t)$ and
	$\lim_{t\to\infty}\dot\theta_i(t)=0$ exist.
\end{conjecture}

\item \textbf{Large inertia and large initial velocity.}
More ambitiously, we pose the following conjecture.
\begin{conjecture}\label{conj:sharpest}
	Given intrinsic velocities $\mathcal{V}\in \mathbb{R}^N$, denote by $\kappa_c(\mathcal{V})$ the critical coupling strength of \eqref{Winfree} computed in \cite[Proposition 63]{ryoo2026oscillator}. Let $\kappa\ge \kappa_c(\mathcal{V})$. Then, for any initial data $(\Theta^0,\Omega^0)\in\R^{2N}$, and any positive inertia $m>0$, the solution $\Theta(t)$ to \eqref{Winfree} exhibits oscillator death.
\end{conjecture}
Of course, Conjecture \ref{conj:sharpest} would imply Conjecture \ref{conj:weak}. Conjecture \ref{conj:sharpest}, if true, would imply that the $m$ does not play a role in asymptotic synchronization; it could, however, affect the effective time required to reach a synchronized state.

In Theorem \ref{thm:pathwise}, the bound on $a$ is needed because the computation of the critical coupling strength $\kappa_c(\mathcal{V})$ in  \cite[Proposition 63]{ryoo2026oscillator} gives $\frac 12 \|\mathcal{V}\|_\infty\le \kappa_c(\mathcal{V})\le \frac{4}{3\sqrt{3}} \|\mathcal{V}\|_\infty$. However, it is unclear what the correct bounds on $b$ and $c$ are, or if they are needed at all. At least for the inertia, numerical results for the inertial Kuramoto model \eqref{eq:Kur-embed}, which carry over to the Winfree model via the embedding of Lemma \ref{lem:embedding}, suggest conflicting results, with some suggesting that low inertia promotes synchronization \cite{dorfler2012synchronization}, while others suggest that low inertia destabilizes  \cite{alberto2000required,chu2010boundary}.

This paper proves that synchronization occurs in the low inertia regime; proving or disproving synchronization in the high inertia regime would require genuinely new tools.
 A sharp pathwise theorem, i.e., one that matches
numerical experiments, such as that of Conjecture \ref{conj:sharpest}, remains to be proved.

\item \textbf{Lyapunov functional.} One way to attack Conjectures \ref{conj:weak} and \ref{conj:sharpest} is to develop a theory of Lyapunov functionals for the Winfree model \eqref{Winfree}. In this direction, we ask whether Question~\ref{conj:lyapunov} has an affirmative answer.
\end{enumerate}

\bigskip
\noindent {\bf Competing interests.} The authors have no relevant financial or non-financial interests to disclose.

\bigskip
\noindent {\bf Data availability statement.} We do not analyze or generate any datasets, because our work proceeds within a theoretical and mathematical approach. One can obtain the relevant materials from the references below.

\bibliographystyle{alpha}
\bibliography{bib}

\end{document}